\documentclass[11pt,twoside, leqno]{article}

\usepackage{amssymb}
\usepackage{amsmath}
\usepackage{amssymb}
\usepackage{amsmath}
\usepackage{amsthm}
\usepackage{color}
\usepackage{mathrsfs}
\usepackage{indentfirst}

\usepackage{txfonts}

\usepackage{anysize}

\allowdisplaybreaks

\newtheorem{theorem}{Theorem}[section]
\newtheorem{lemma}[theorem]{Lemma}
\newtheorem{corollary}[theorem]{Corollary}
\newtheorem{proposition}[theorem]{Proposition}
\theoremstyle{definition}

\newcounter{assum}

\newtheorem{assumption}[assum]{Assumption}
\newtheorem{remark}[theorem]{Remark}
\newtheorem{definition}[theorem]{Definition}
\numberwithin{equation}{section}

\begin{document}

\title{\Large\bf Global BMO-Sobolev Estimates for Second-Order
Linear Elliptic Equations on Lipschitz Domains
\footnotetext{\hspace{-0.35cm} 2020 {\it Mathematics Subject
Classification}. {Primary 35J25; Secondary 35B45, 35B65, 42B35, 42B37.}
\endgraf{\it Key words and phrases}. second-order elliptic equation, global regularity estimate,
Dirichlet problem, Neumann problem, Robin problem, BMO space.
\endgraf H. Dong is partially supported by the NSF under agreement DMS-2350129. D. Yang is partially supported by the National Key Research and
Development Program of China (Grant No. 2020YFA0712900) and the National Natural Science Foundation of China
(Grant Nos. 12431006 and 12371093). S. Yang is partially supported by the National Natural Science Foundation of China
(Grant Nos. 12431006 and 12071431), the Key Project of the Gansu Provincial National Science Foundation (Grant No. 23JRRA1022),
the Fundamental Research Funds for the Central Universities (Grant No. lzujbky-2021-ey18)
and the Innovative Groups of Basic Research in Gansu Province (Grant No. 22JR5RA391).}}

\author{Hongjie Dong\footnote{Corresponding author, E-mail:
\texttt{Hongjie\_Dong@Brown.edu}/{\color{red} September 27, 2024}/Final version.},\ Dachun Yang and Sibei Yang}
\date{}
\maketitle

\vspace{-0.8cm}

\begin{center}
\begin{minipage}{13.5cm}\small
{{\bf Abstract.}
Let $n \ge 2$ and $\Omega \subset \mathbb{R}^n$ be a bounded Lipschitz domain. In this article, we establish first-order global regularity estimates in the scale of BMO spaces on $\Omega$ for weak solutions to the second-order elliptic equation $\mathrm{div}(A \nabla u) = \mathrm{div} , \boldsymbol{f}$ in $\Omega$. This is achieved under minimal regularity assumptions on $\Omega$ and the coefficient matrix $A$, utilizing the pointwise multiplier characterization of the BMO space on $\Omega$. As an application, we also obtain global estimates of $\nabla u$ in the Lebesgue space $L^1(\Omega)$ when $\boldsymbol{f}$ belongs to the Hardy space on $\Omega$.}
\end{minipage}
\end{center}

\vspace{0.1cm}

\section{Introduction and main results\label{s1}}

Let $n\ge2$ and $\Omega\subset{\mathbb{R}^n}$ be a bounded Lipschitz domain. In this article,
we study the second-order elliptic equation in divergence form
\begin{equation}\label{e1.1}
\mathrm{div}(A\nabla u)=\mathrm{div}\boldsymbol{f}\ \text{in}\ \Omega,
\end{equation}
with the Dirichlet, the Neumann, or the Robin boundary condition. With minimal regularity
assumptions on $\Omega$ and the coefficient matrix $A$ (see Assumption \ref{a1.1} for the details),
we derive global estimates for $\nabla u$ in the scale of BMO spaces on $\Omega$. As applications, we also establish
the global estimate for $\nabla u$ in the Lebesgue space $L^1(\Omega)$ when
$\boldsymbol{f}$ belongs to the Hardy space on $\Omega$. The global regularity estimates
obtained in this article are natural extensions of the known global Calder\'on--Zygmund type
estimate
\begin{equation}\label{e1.2}
\|\nabla u\|_{L^p(\Omega;{\mathbb{R}^n})}\le C\|\boldsymbol{f}\|_{L^p(\Omega;{\mathbb{R}^n})},
\end{equation}
with $p\in(1,\infty)$, where $C$ is a positive constant independent of $u$ and $\boldsymbol{f}$. Our work extend this estimate to the endpoint cases of $p=\infty$ and $p=1$.

To state the main results of this article and related background, we first recall
several necessary concepts and notation. Let $n\ge2$, $\Omega\subset{\mathbb{R}^n}$ be a domain, and $p\in[1,\infty]$.
Recall that the \emph{Lebesgue space $L^p(\Omega)$} is defined to be the set of all measurable
functions $f$ on $\Omega$ satisfying
\begin{equation}\label{e1.3}
\|f\|_{L^p(\Omega)}
:=\begin{cases}
\left[\displaystyle\int_\Omega |f(x)|^p\,dx\right]^{\frac1p}<\infty, \quad & p\in[1,\infty),\\
{\displaystyle\mathop{\rm ess\,sup}_{x\in\Omega}}\,|f(x)|<\infty, \quad & p=\infty,
\end{cases}
\end{equation}
where ${\rm ess\,sup}_{x\in\Omega}{|f(x)|}$ denotes the \emph{essential supremum} of $|f|$ on $\Omega$.
Moreover, for any given $m\in\mathbb{N}$, let
\begin{equation}\label{e1.4}
L^p(\Omega;\mathbb{R}^m):=\left\{\boldsymbol{f}:=(f_1,\ldots,f_m):\ \text{for any}
\ i\in\{1,\ldots,m\},\ f_i\in L^p(\Omega)\right\}
\end{equation}
with
$$\|\boldsymbol{f}\|_{L^p(\Omega;\mathbb{R}^m)}:=\sum_{i=1}^m\|f_i\|_{L^p(\Omega)}.
$$
Additionally, we denote by $W^{1,p}(\Omega)$ the \emph{Sobolev space on $\Omega$}, equipped with the
norm:
$$
\|f\|_{W^{1,p}(\Omega)}:=\|f\|_{L^p(\Omega)}+\|\nabla f\|_{L^p(\Omega;{\mathbb{R}^n})},
$$
where $\nabla f:=(f_{x_1},\ldots,f_{x_n})$ is the \emph{gradient} of $f$ and
$\{f_{x_i}\}_{i=1}^n$ are the distributional derivatives of $f$. Furthermore, $W^{1,p}_{0}(\Omega)$
is defined to be the \emph{closure} of $C^{\infty}_{\mathrm{c}} (\Omega)$ in $W^{1,p}(\Omega)$, where
$C^{\infty}_{\mathrm{c}}(\Omega)$ denotes the set of all infinitely differentiable functions on $\Omega$
with compact support contained in $\Omega$.

We assume that the matrix $A:=\{a_{i,j}\}_{i,j=1}^n$ is real-valued, bounded,
and measurable and satisfies the \emph{uniform ellipticity condition}, that is, there exists a positive
constant $\mu_0\in(0,1]$ such that, for any $x\in\Omega$ and $\xi:=(\xi_1,\ldots,\xi_n)\in{\mathbb{R}^n}$,
\begin{equation}\label{e1.5}
\mu_0|\xi|^2\le\sum_{i,j=1}^na_{i,j}(x)\xi_i\xi_j\le \mu_0^{-1}|\xi|^2.
\end{equation}
Let $n\ge2$, $\Omega\subset{\mathbb{R}^n}$ be a bounded Lipschitz domain, $p\in[1,\infty]$, and
$\boldsymbol{f}\in L^p(\Omega;{\mathbb{R}^n})$. Denote by $\partial\Omega$ the boundary of $\Omega$ and
$\boldsymbol{\nu}:=(\nu_1,\ldots,\nu_n)$ the \emph{outward unit normal}
to $\partial\Omega$. A function $u$ is called a \emph{weak solution} of the Neumann problem
\begin{equation}\label{e1.6}
\left\{\begin{array}{ll}
\mathrm{div}(A\nabla u)=\mathrm{div} \boldsymbol{f}\ \ &\text{in}\ \Omega,\\
\displaystyle\frac{\partial u}{\partial\boldsymbol{\nu}}=\boldsymbol{f}\cdot\boldsymbol{\nu}\  &\text{on}\  \partial\Omega
\end{array}\right.
\end{equation}
if $u\in W^{1,p}(\Omega)$ and, for any $\varphi\in C^{\infty}({\mathbb{R}^n})$ (the set of all infinitely differentiable
functions on ${\mathbb{R}^n}$),
\begin{align}\label{e1.7}
\int_{\Omega}A(x)\nabla u(x)\cdot \nabla\varphi(x)\,dx=\int_{\Omega}\boldsymbol{f}(x)\cdot\nabla\varphi(x)\,dx.
\end{align}
Here and thereafter, $\dfrac{\partial u}{\partial\boldsymbol{\nu}}:=(A\nabla u)\cdot\boldsymbol{\nu}$ denotes
the \emph{conormal derivative} of $u$ on $\partial\Omega$. The Neumann problem \eqref{e1.6}
is said to be \emph{uniquely solvable} if, for any given $\boldsymbol{f}\in L^p(\Omega;{\mathbb{R}^n})$,
there exists $u\in W^{1,p}(\Omega)$, unique up to a constant, such that \eqref{e1.7} holds.
A function $u$ is called a \emph{weak solution} of the Dirichlet problem
\begin{equation}\label{e1.8}
\left\{\begin{array}{ll}
\mathrm{div}(A\nabla u)=\mathrm{div} \boldsymbol{f}\ \ &\text{in}\ \Omega,\\
u=0\  &\text{on}\  \partial\Omega
\end{array}\right.
\end{equation}
if $u\in W^{1,p}_0(\Omega)$ and \eqref{e1.7} holds for any $\varphi\in C_{\rm c}^{\infty}(\Omega)$.
The Dirichlet problem \eqref{e1.8} is said to be \emph{uniquely solvable}
if, for any given $\boldsymbol{f}\in L^{p}(\Omega;{\mathbb{R}^n})$, there exists a unique $u\in W^{1,p}_0(\Omega)$
such that \eqref{e1.7} holds for any $\varphi\in C^\infty_{\rm c}(\Omega)$.

The global regularity theory of (non-)linear elliptic equations (or systems) in non-smooth domains
is a central and compelling area of research in partial differential equations (see, for instance,
\cite{bw04,cm14,d20,k94,sh18}). For the Dirichlet problem \eqref{e1.8}, the global
Calder\'on--Zygmund type estimate \eqref{e1.2} was obtained in \cite{d96} for any $p\in(1,\infty)$ under the assumptions that $A\in\mathrm{VMO}({\mathbb{R}^n};\mathbb{R}^{n^2})$ (see, for instance, \cite{s75} for the definition of the VMO space) and $\partial\Omega\in C^{1,1}$, the latter of which was then weakened to $\partial\Omega\in C^{1}$
in \cite{aq02}. Additionally, for any given $p\in(1,\infty)$, the estimate \eqref{e1.2} was established in \cite{b05,bw04} for the Dirichlet problem \eqref{e1.8}, under the assumptions that $A$ satisfies the $(\delta,R)$-BMO
condition (see, for instance, \cite{bw04} for the definition of the $(\delta,R)$-BMO
condition) for sufficiently small $\delta\in(0,\infty)$ and that $\Omega$ is a bounded Lipschitz domain
with small Lipschitz constant or a bounded Reifenberg flat domain (see, for instance, \cite{bw04,bw05} for
the definition of the Reifenberg flat domain).
For the Dirichlet problem \eqref{e1.8} with partial small $\mathrm{BMO}$ coefficients,
the estimate \eqref{e1.2} with any given $p\in(1,\infty)$ was systematically studied in \cite{dk10,k07}, under the assumption that $\Omega$ is a bounded Lipschitz domain with small Lipschitz constant.
Meanwhile, for the problem \eqref{e1.8} in a general Lipschitz domain $\Omega$,
it was proved in \cite{sh05a} that, if $A$ is symmetric and $A\in\mathrm{VMO}({\mathbb{R}^n};\mathbb{R}^{n^2})$, then
\eqref{e1.2} holds for any $p\in(\frac32-\varepsilon,3+\varepsilon)$ when $n\ge3$ or $p\in(\frac43-\varepsilon,4+\varepsilon)$ when $n=2$,
where $\varepsilon\in(0,\infty)$ is a positive constant depending only on the Lipschitz constant of $\Omega$ and $n$.
The range of $p$ obtained in \cite{sh05a} is sharp for general Lipschitz domains (see \cite{sh05a}
for the details). We also refer to \cite{dek18,dk18,dk17} for more recent progress
on the global regularity estimate of the Dirichlet problem \eqref{e1.8}.

For the Neumann problem \eqref{e1.6}, the estimate \eqref{e1.2} was proved in \cite{aq02}
for any $p\in(1,\infty)$, under the assumptions that $A\in\mathrm{VMO}({\mathbb{R}^n};\mathbb{R}^{n^2})$ and $\partial\Omega\in C^1$.
Furthermore, for any given $p\in(1,\infty)$, when $A$ has small $\mathrm{BMO}$ coefficients and $\Omega$ is a bounded
Reifenberg flat domain or $A$ has partial small $\mathrm{BMO}$ coefficients and $\Omega$ is a bounded Lipschitz
domain with small Lipschitz constant or a bounded Reifenberg flat domain, the estimate \eqref{e1.2} was established,
respectively, in \cite{bw05} and \cite{dk10,dk12} for the Neumann problem \eqref{e1.6}.
For the Neumann problem \eqref{e1.6} on a general Lipschitz domain, it was proved in \cite{g12} that,
if $A$ is symmetric and $A\in\mathrm{VMO}({\mathbb{R}^n};\mathbb{R}^{n^2})$, then \eqref{e1.2} holds for any given
$p\in(\frac32-\varepsilon,3+\varepsilon)$ when $n\ge3$ or $p\in(\frac43-\varepsilon,4+\varepsilon)$ when $n=2$,
where $\varepsilon\in(0,\infty)$ is a positive constant depending only on the Lipschitz constant of $\Omega$ and $n$.
It is also worth pointing out that the range of $p$ such that \eqref{e1.2} holds obtained in \cite{g12} is sharp
for general Lipschitz domains. We refer to \cite{dlk20,dk18,ycyy20} for more results
on regularity estimates of the Neumann problem \eqref{e1.6}.

The global BMO estimate of the gradient for the weak solution to the Dirichlet problem \eqref{e1.8}
was established in \cite{a92} under the assumptions that the matrix $A$ satisfies a log-type BMO condition
and $\partial\Omega\in C^{1+\alpha}$ with some $\alpha\in(0,\infty)$. We also point out that the local and the global BMO estimates of the gradient to the weak solution of $p$-Laplace equations were studied in \cite{dks12,bc22}.
Global $C^1$ and weak-type $(1,1)$ estimates for the problem \eqref{e1.6} or \eqref{e1.8} were obtained
in \cite{dek18,dlk20,l17} under certain Dini continuity assumptions on the matrix $A$ and the domain $\Omega$.

In this paper, we demonstrate that, under minimal regularity assumptions on $A$ and $\partial \Omega$, the global Calder\'on--Zygmund type estimate \eqref{e1.2}, with an appropriate modified version, remains valid in the endpoint cases $p=\infty$ and $p=1$ for both the Dirichlet problem \eqref{e1.8} and the Neumann
problem \eqref{e1.6},

To state the main results of this article, we begin by recalling several concepts on Campanato type spaces and BMO type spaces on domains.

In the following, for any $x\in{\mathbb{R}^n}$ and $r\in(0,\infty)$, we define $B(x,r):=\{y\in{\mathbb{R}^n}:\ |y-x|<r\}$.
Let $\Omega\subset{\mathbb{R}^n}$ be a domain. Denote by $L^1_{\mathrm{loc}}(\Omega)$ the \emph{set of all locally integrable
functions on $\Omega$}.

\begin{definition}\label{d1.1}
Let $n\ge2$, $\Omega\subset{\mathbb{R}^n}$ be a domain, and $p\in[1,\infty)$,
and let $\omega:\ [0,\infty)\to[0,\infty)$ be a continuous and non-decreasing function. The \emph{Campanato type space} $\mathcal{L}^{\omega(\cdot),p}(\Omega)$ is defined to be the set
of all $f\in L^1_{\mathrm{loc}}(\Omega)$ satisfying
\begin{equation*}
\|f\|_{\mathcal{L}^{\omega(\cdot),p}(\Omega)}:=\sup_{x\in\Omega,r\in(0,\mathrm{diam\,}(\Omega))}
\frac{1}{\omega(r)}\left[\fint_{B(x,r)\cap\Omega}\left|f(y)-(f)_{B(x,r)\cap\Omega}\right|^p\,dy\right]^{\frac1p}<\infty.
\end{equation*}
Here and thereafter, $\mathrm{diam\,}(\Omega):=\sup\{|x-y|:\ x,y\in\Omega\}$ and, for any measurable set $E\subset \Omega$ with $|E|<\infty$ and
locally integrable (vector-valued) function $g$ on $\Omega$,
$$(g)_E:=\fint_{E}g(y)\,dy:=\frac{1}{|E|}\int_E g(y)\,dy.$$

When $p=1$, the space $\mathcal{L}^{\omega(\cdot),p}(\Omega)$ is simply denoted by $\mathcal{L}^{\omega(\cdot)}(\Omega)$. When $\omega\equiv1$, the space $\mathcal{L}^{\omega(\cdot)}(\Omega)$ is the space of functions of \emph{bounded mean oscillation} on $\Omega$ (see, for instance, \cite{bc22,n08,n97}), and is denoted by $\mathrm{BMO}(\Omega)$.
\end{definition}

We note that, under mild assumptions on $\omega$ and $\Omega$, for any given $p\in[1,\infty)$, the spaces
$\mathcal{L}^{\omega(\cdot),p}(\Omega)$ and $\mathcal{L}^{\omega(\cdot)}(\Omega)$ are equivalent (see, for instance, \cite[Theorem 3.1]{n08} or Lemma \ref{l2.1}).

Furthermore, for a bounded open set $\Omega$ of ${\mathbb{R}^n}$, if, in a neighborhood of each point of
$\partial\Omega$, $\partial\Omega$ agrees with the subgraph of a function $\psi$ of $(n-1)$ variables that
belongs to the function space $X$, then we \emph{write} $\partial\Omega\in X$. Similarly, the notation
$\partial\Omega\in W^1X$ means that such function $\psi$ is weakly differentiable and its weak derivatives
belong to the space $X$.

\begin{definition}\label{d1.2}
Let $n\ge2$ and $\Omega\subset{\mathbb{R}^n}$ be a bounded domain.
\begin{itemize}
\item[{\rm(i)}] Let $f\in L^{1}_{\mathrm{loc}}({\mathbb{R}^n})$. Then, $f$ is said to belong to the \emph{space} $\mathrm{BMO}({\mathbb{R}^n})$ if
$$\|f\|_{\mathrm{BMO}({\mathbb{R}^n})}:=\sup_{B\subset{\mathbb{R}^n}}\fint_{B}\left|f(x)-(f)_B\right|\,dx<\infty,$$
where the supremum is taken over all balls $B\subset{\mathbb{R}^n}$.

\item[{\rm(ii)}] The ``restricted type" $\mathrm{BMO}$ space $\mathrm{BMO}_r(\Omega)$ on $\Omega$ is defined by setting
$$\mathrm{BMO}_r(\Omega):=\left\{f\in L^{1}_{\mathrm{loc}}(\Omega):\ \text{there exists}\ F\in \mathrm{BMO}({\mathbb{R}^n})\
\text{such that}\  F|_{\Omega}=f\right\}.$$
For any $f\in \mathrm{BMO}_r(\Omega)$, define
$$\|f\|_{\mathrm{BMO}_r(\Omega)}:=\inf\left\{\|F\|_{\mathrm{BMO}({\mathbb{R}^n})}:\ F\in \mathrm{BMO}({\mathbb{R}^n})\ \text{and}\ F|_{\Omega}=f\right\}$$
and
$$\|f\|_{\mathrm{BMO}_{r,+}(\Omega)}:=\|f\|_{\mathrm{BMO}_r(\Omega)}+\|f\|_{L^2(\Omega)}.$$

For any given $m\in\mathbb{N}$, the space $\mathrm{BMO}_r(\Omega;\mathbb{R}^m)$ is defined via
replacing $L^p(\Omega)$ in \eqref{e1.3} by the aforementioned $\mathrm{BMO}_r(\Omega)$ in the
definition of $L^p(\Omega;\mathbb{R}^m)$ in \eqref{e1.4}.
\end{itemize}
\end{definition}

It is worth mentioning that the spaces $\mathrm{BMO}(\Omega)$ and $\mathrm{BMO}_r(\Omega)$
are suitable replacements for the Lebesgue space $L^\infty(\Omega)$ when studying the boundedness of certain operators
or the well-posedness problems of certain partial differential equations (see, for instance, \cite{a92,bg18,cds99,gg22,g14a,St93}).

\begin{remark}\label{r1.1}
Let $n\ge2$ and $\Omega$ be a bounded Lipschitz domain of ${\mathbb{R}^n}$.
\begin{itemize}
\item[{\rm(i)}] When $p=2$, by the Lax--Milgram theorem (see, for instance, \cite[Section 1.3.1, Lemma 3.1]
{j13}), we know that the Neumann problem \eqref{e1.6} and the Dirichlet problem \eqref{e1.8} are uniquely
solvable and the estimate \eqref{e1.2} holds. Meanwhile, for the Dirichlet problem \eqref{e1.8},
from the divergence theorem, it follows that, for any $\boldsymbol{f}_0\in{\mathbb{R}^n}$ and $\varphi\in C^\infty_{\rm c}(\Omega)$,
\begin{align}\label{e1.9}
\int_{\Omega}A(x)\nabla u(x)\cdot \nabla\varphi(x)\,dx=\int_{\Omega}[\boldsymbol{f}(x)-\boldsymbol{f}_0]\cdot\nabla\varphi(x)\,dx.
\end{align}
Thus, for any given $p\in(1,\infty)$, if \eqref{e1.2} holds for the Dirichlet problem \eqref{e1.8}, then the estimate \eqref{e1.2} also
holds for the problem \eqref{e1.8} with $\boldsymbol{f}$ replaced by $\boldsymbol{f}-\boldsymbol{f}_0$.

When $p\in(1,\infty)$ and $p\neq2$, the Neumann problem \eqref{e1.6} and the Dirichlet problem \eqref{e1.8}
may not be uniquely solvable (see, for instance, \cite[p.\,1285]{bw04}). Some extra conditions on
both the domain $\Omega$ and the matrix $A$ are necessary to guarantee the unique solvability of the
Neumann problem \eqref{e1.6} and the Dirichlet problem \eqref{e1.8} when $p\neq2$ (see, for instance,
\cite{bw05,bw04,d20,dk10,g12,sh05a}).
\item[{\rm(ii)}] By Lemmas \ref{l2.1} and \ref{l2.2}, we conclude that
$\mathrm{BMO}_r(\Omega)\subset L^p(\Omega)$ for any $p\in(1,\infty)$ as sets. Moreover,
it is easy to find that $L^\infty(\Omega)\subset\mathrm{BMO}_r(\Omega)$. Thus, when
$\boldsymbol{f}\in\mathrm{BMO}_r(\Omega)$, the weak solution of the Neumann problem \eqref{e1.6} or
the Dirichlet problem \eqref{e1.8} uniquely exists in $W^{1,2}(\Omega)$ or $W^{1,2}_0(\Omega)$.
\end{itemize}
\end{remark}

To state the main result of this article, we also need an assumption on the matrix $A$
and the domain $\Omega$ as follows.

\begin{assumption}\label{a1.1}
Assume that there exist a constant $R_0\in(0,\infty)$ and a
function $\sigma:\ [0,\infty)\to[0,\infty)$ such that
the matrix $A:=\{a_{i,j}\}_{i,j=1}^n$ and the domain $\Omega$ satisfy the following conditions:
\begin{itemize}
\item[(a)]
\begin{equation*}
\sum_{i,\,j=1}^n\sup_{x\in\Omega,\, r\in(0,R_0)}\frac{1}{\sigma(r)}\fint_{B(x,r)\cap\Omega}\left|a_{i,j}(y)-
\fint_{B(x,r)\cap\Omega}a_{i,j}(z)\,dz\right|\,dy<\infty ,
\end{equation*}
where the function $\sigma$ satisfies
\begin{itemize}
\item[(i)] $\lim_{r\to 0^+}\sigma(r)\ln({\frac{1}{r}})=0$, where $r\to 0^+$ \emph{means} that $r\in(0,R_0)$ and $r\to0$;
\item[(ii)] for any $s,r\in(0,R_0)$, if $C_1^{-1}s<r<C_1s$, then $C_2^{-1}\sigma(s)\le\sigma(r)\le C_2\sigma(s)$,
where $C_1$ and $C_2$ are positive constants independent of $s$ and $r$;
\item[(iii)] there exists a positive constant $C$ such that, for any $s,r\in(0,R_0)$ satisfying $s\le r$,
$\sigma(s)\le C\sigma(r)$.
\end{itemize}
\item[(b)]  $\partial\Omega\in W^1\mathcal{L}^{\sigma(\cdot)}$ with $\sigma$ being the same as in (a).
\end{itemize}
\end{assumption}

\begin{remark}
In this remark, we show that starting from a modulus of continuity $\sigma$, which only satisfies (i) and is bounded on $(0,R_0]$, we can construct another modulus of continuity $\tilde \sigma\ge \sigma$, which satisfies all conditions (i)-(iii).
Without loss of generality, by considering $\sup_{0\le s\le r}\sigma(s)$ instead of $\sigma(r)$ we may assume that $\sigma$ is nondecreasing. Now we define $\tilde \sigma(r):=\sup_{r\le s\le R_0}r\sigma(s)/s$ for $r\in (0,1)$. Then using the factor that $r\ln (1/r)$ is an increasing function on $(0,r_0)$ for small $r_0$, it is easily seen that $\tilde\sigma$ satisfies (i). Since $\sigma$ is nondecreasing, for any $r_1,r_2\in (0,R_0)$ satisfying $r_1<r_2$, we have
\begin{align*}
\tilde \sigma(r_1)&=\sup_{r_1\le s\le R_0}r_1\sigma(s)/s\le \max\{\sup_{r_1\le s\le r_2}r_1\sigma(s)/s, \sup_{r_2\le s\le R_0}r_1\sigma(s)/s\}\\
&\le\max\{\sup_{r_1\le s\le r_2}r_1\sigma(r_2)/s, \sup_{r_2\le s\le R_0}r_2\sigma(s)/s\}\le \tilde \sigma(r_2).
\end{align*}
Thus, $\tilde \sigma$ also satisfies (iii). Finally, the above inequality together with the fact that $\tilde \sigma(r)/r$ is non-increasing implies (ii).
\end{remark}

Now we state the main results of this article.

\begin{theorem}\label{th1.1}
Let $n\ge2$ and $\Omega\subset{\mathbb{R}^n}$ be a bounded Lipschitz domain. Assume that $A$ and $\Omega$ satisfy
Assumption \ref{a1.1}. Let $u\in W^{1,2}(\Omega)$ be the weak solution to the Neumann problem \eqref{e1.6} or the Dirichlet
problem \eqref{e1.8} with $\boldsymbol{f}\in\mathrm{BMO}_{r}(\Omega;{\mathbb{R}^n})$. Then $\nabla u\in
\mathrm{BMO}_{r}(\Omega;{\mathbb{R}^n})$ and there exists a positive constant $C$ independent of  $u$ and $\boldsymbol{f}$
such that
\begin{equation}\label{e1.10}
\left\|\nabla u\right\|_{\mathrm{BMO}_{r,+}(\Omega;{\mathbb{R}^n})}\le C\left\|\boldsymbol{f}\right\|_{\mathrm{BMO}_{r,+}(\Omega;{\mathbb{R}^n})}.
\end{equation}
\end{theorem}

\begin{remark}
\begin{itemize}
\item[{\rm(i)}] From the proof of Theorem \ref{th1.1}, we deduce that, for the Dirichlet problem \eqref{e1.8}, the estimate
\eqref{e1.10} can be reinforced to
\begin{equation}\label{e1.10a}
\|\nabla u\|_{\mathrm{BMO}_{r,+}(\Omega;{\mathbb{R}^n})}\le C
\|\boldsymbol{f}\|_{\mathrm{BMO}_{r}(\Omega;{\mathbb{R}^n})}.
\end{equation}
However, the estimate \eqref{e1.10a} may \emph{not} hold for the Neumann problem
\eqref{e1.6} even when both $\Omega$ and $A$ are smooth. For example, let $B_0:=B(\mathbf{0},1)$ be a ball of $\mathbb{R}^2$,
where $\mathbf{0}$ denotes the \emph{origin} of $\mathbb{R}^{2}$. For any
$(x_1,x_2)\in B_0$, let
$$
A(x_1,x_2):=\begin{pmatrix}\dfrac{1}{1+x_1^2}
  & 0\\0
  &\dfrac{1}{1+x_2^2}
\end{pmatrix}\quad \text{and}\quad \boldsymbol{f}(x_1,x_2):=(1,1).
$$
Then $u(x_1,x_2):=(x_1+x_2)+(x_1^3+x_2^3)/3$ is a weak
solution of the Neumann problem \eqref{e1.6} with $\Omega:=B_0$. Obviously,
in this case, the coefficient matrix $A$ and the domain $B_0$ satisfy
Assumption \ref{a1.1}, $\|\nabla u\|_{\mathrm{BMO}_{r}(B_0;\mathbb{R}^2)}>0$, and
$\|\boldsymbol{f}\|_{\mathrm{BMO}_{r}(B_0;\mathbb{R}^2)}=0$. Thus, the estimate \eqref{e1.10a}
fails in this case.
\item[{\rm(ii)}] By \cite[Remark 5.3]{a92} and \cite[Theorem 2.2]{bc22}, we find that
Assumption \ref{a1.1} on the matrix $A$ and the domain $\Omega$ in Theorem \ref{th1.1} is \emph{sharp}
to guarantee  that the estimate \eqref{e1.10} holds.
\item[{\rm(iii)}] Recall that the global estimate \eqref{e1.10} was established in \cite[Theorem 2.2]{a92} for
the Dirichlet problem \eqref{e1.8} under the assumption that $A$ satisfies Assumption \ref{a1.1}(a) and the domain $\Omega$
is bounded and satisfies $\partial\Omega\in C^{1+\alpha}$ with some $\alpha\in(0,\infty)$. It is easy to find that, if the bounded
domain $\Omega$ satisfies $\partial\Omega\in C^{1+\alpha}$ with some $\alpha\in(0,\infty)$, then $\Omega$ satisfies Assumption \ref{a1.1}(b).
Thus, Theorem \ref{th1.1} improves \cite[Theorem 2.2]{a92} by weakening the assumption on the domain $\Omega$.

Moreover, the assertion of Theorem \ref{th1.1} in the case of the Neumann problem \eqref{e1.6} is new even when
the domain $\Omega$ satisfies $\partial\Omega\in C^{1+\alpha}$ with some $\alpha\in(0,\infty)$.
\end{itemize}
\end{remark}

We prove Theorem \ref{th1.1} by establishing mean oscillation-type estimates
for $\nabla u$ and utilizing the equivalent characterization of the
space $\mathrm{BMO}_r(\Omega)$ (see Lemma \ref{l2.2}) and the pointwise multiplier characterization
of the space $\mathrm{BMO}(\Omega)$ (see, for instance, \cite{j76,n97,ny97,ny85} or Lemma \ref{l2.4}).
To derive the mean oscillation estimate of $\nabla u$ in the interior of $\Omega$
(see Proposition \ref{p5.1}), Assumption \ref{a1.1}(a) for the matrix $A$ is required.

The main part of the proof of Theorem \ref{th1.1} is to establish the mean oscillation estimate
of $\nabla u$ near the boundary of $\Omega$ (see Theorems \ref{t5.1} and \ref{t5.2}).
To achieve this, we employ a flattening technique and the pointwise multiplier characterization
of the space $\mathrm{BMO}(\Omega)$. In this part,
both Assumption \ref{a1.1}(a) on the matrix $A$ and Assumption \ref{a1.1}(b) on the domain $\Omega$ are used.

Next, we recall the definitions of the Hardy space $H^1({\mathbb{R}^n})$ and the ``supported type" Hardy space
$H^1_z(\Omega)$.

\begin{definition}\label{d1.3}
Let $n\ge2$ and $\Omega$ be a bounded Lipschitz domain of ${\mathbb{R}^n}$, and let $\phi\in C^\infty_{\rm c}({\mathbb{R}^n})$
be non-negative and $\int_{{\mathbb{R}^n}}\phi(x)\,dx=1$. For any $t\in(0,\infty)$ and $x\in{\mathbb{R}^n}$, define $\phi_t(x):=t^{-n}\phi(x/t)$.
A function $f\in L^1({\mathbb{R}^n})$ is said to be in the \emph{Hardy space} $H^1({\mathbb{R}^n})$ if $\mathcal{M}(f):=\sup_{t\in(0,\infty)}
|\phi_t\ast f|\in L^1({\mathbb{R}^n})$.
Let
$$\|f\|_{H^1({\mathbb{R}^n})}:=\left\|\mathcal{M}(f)\right\|_{L^1({\mathbb{R}^n})}.$$
The ``supported type" \emph{Hardy space $H^1_{z}(\Omega)$} is defined by setting
\begin{equation*}
H^1_{z}(\Omega):=\left\{f\in H^1({\mathbb{R}^n}):\ \mathrm{supp\,}(f)\subset\overline{\Omega}\right\},
\end{equation*}
where $\overline{\Omega}$  denote the \emph{closure} of $\Omega$ in ${\mathbb{R}^n}$. Moreover, for any
$f\in H^1_{z}(\Omega)$, let $\|f\|_{H^1_{z}(\Omega)}:=\|f\|_{H^1({\mathbb{R}^n})}.$
\end{definition}

Like the spaces $\mathrm{BMO}(\Omega)$ and $\mathrm{BMO}_r(\Omega)$, the Hardy space
$H^1({\mathbb{R}^n})$ or $H^1_{z}(\Omega)$ is respectively a suitable replacement of the Lebesgue space $L^1({\mathbb{R}^n})$ or $L^1(\Omega)$
(see, for instance, \cite{C94,cds99,cks93,g14a,St93}).

As an application of Theorem \ref{th1.1} and using the fact that $\mathrm{BMO}_{r}(\Omega)$ is the dual space
of the Hardy space $H^1_{z}(\Omega)$ (see, for instance, Lemma \ref{l2.3}), we obtain the following
global regularity estimate in $L^1(\Omega;{\mathbb{R}^n})$ for the problems \eqref{e1.6} and \eqref{e1.8}.

\begin{corollary}\label{c1.1}
Let $n\ge2$ and $\Omega\subset{\mathbb{R}^n}$ be a bounded Lipschitz domain. Assume that $A$ and $\Omega$ satisfy Assumption \ref{a1.1}. Then the Neumann problem \eqref{e1.6} or the Dirichlet problem \eqref{e1.8} with
$\boldsymbol{f}\in H^1_{z}(\Omega;{\mathbb{R}^n})$ is uniquely solvable and the weak solution
$u$ belongs to $W^{1,1}(\Omega)$ or $W^{1,1}_0(\Omega)$, respectively. Moreover, there exists a positive constant
$C$ independent of $u$ and $\boldsymbol{f}$ such that
$$
\|\nabla u\|_{L^1(\Omega;{\mathbb{R}^n})}
\le C\|\boldsymbol{f}\|_{H^1_z(\Omega;{\mathbb{R}^n})}.
$$
\end{corollary}

Recall that the global estimate \eqref{e1.2} for the problem \eqref{e1.6} or \eqref{e1.8}
in the scale of Lebesgue spaces $L^p(\Omega)$ with any given $p\in(1,\infty)$ holds under some mild assumptions
on $A$ and $\Omega$. Moreover, global $C^1$ and weak-type
$(1,1)$ estimates for the problem \eqref{e1.6} or \eqref{e1.8} were obtained in \cite{dek18,dlk20} under the
Dini mean oscillation condition on $A$ and the $C^{1,\text{Dini}}$ condition on $\Omega$, which are somewhat stronger than those assumptions
on $A$ and $\Omega$ in Theorem \ref{th1.1} and Corollary \ref{c1.1}. Thus, the endpoint type global estimates given
in Theorem \ref{th1.1} and Corollary \ref{c1.1} can be seen as an intermediate case between the global estimate \eqref{e1.2}
in the scale of Lebesgue spaces $L^p(\Omega)$ with any given $p\in(1,\infty)$ and the global $C^1$ estimate for the problem
\eqref{e1.6} or \eqref{e1.8}.

Our last result is regarding the Robin problem. Let $n\ge2$, $\Omega\subset{\mathbb{R}^n}$ be a bounded Lipschitz domain, and $d\sigma$ be the \emph{surface measure}
on $\partial\Omega$. Assume that $\beta$ is a measurable function on $\partial\Omega$ satisfying that
\begin{equation}\label{e1.11}
0\le\beta\in L^{\infty}(\partial\Omega)\ \ \text{and}\ \ \beta\ge c_0\ \ \text{on}\ \
E_0\subset\partial\Omega,
\end{equation}
where $c_0\in(0,\infty)$ is a given constant and the measurable set $E_0$ satisfies $\sigma(E_0)>0$.
Let $p\in[1,\infty]$ and $\boldsymbol{f}\in L^p(\Omega;{\mathbb{R}^n})$. A function $u$ is called a \emph{weak solution} of the Robin problem
\begin{equation}\label{e1.12}
\left\{\begin{array}{ll}
\mathrm{div}(A\nabla u)=\mathrm{div} \boldsymbol{f}\ \ &\text{in}\ \Omega,\\
\dfrac{\partial u}{\partial\boldsymbol{\nu}}+\beta u=\boldsymbol{f}\cdot\boldsymbol{\nu}\
&\text{on}\  \partial\Omega
\end{array}\right.
\end{equation}
if $u\in W^{1,p}(\Omega)$ and, for any $\varphi\in C^{\infty}({\mathbb{R}^n})$,
\begin{align}\label{e1.13}
\int_{\Omega}A(x)\nabla u(x)\cdot \nabla\varphi(x)\,dx+\int_{\partial\Omega}\beta(x)u(x)\varphi(x)\,d\sigma(x)
=\int_{\Omega}\boldsymbol{f}(x)\cdot\nabla\varphi(x)\,dx.
\end{align}
The Robin problem \eqref{e1.12} is said to be \emph{uniquely solvable} if, for any given $\boldsymbol{f}\in L^p(\Omega;{\mathbb{R}^n})$,
there exists a unique $u\in W^{1,p}(\Omega)$ such that \eqref{e1.13} holds.
It is known that, when $p=2$, the Robin problem \eqref{e1.12} is uniquely solvable (see Remark \ref{r6.1}).

Applying Theorem \ref{th1.1} and a perturbation method, we obtain the following global regularity estimate
for the Robin problem \eqref{e1.12} in both $\mathrm{BMO}_{r}(\Omega;{\mathbb{R}^n})$ and $L^1(\Omega;{\mathbb{R}^n})$.

\begin{theorem}\label{th1.2}
Let $n\ge2$, $\Omega\subset{\mathbb{R}^n}$ be a bounded Lipschitz domain, and $\beta$ satisfies \eqref{e1.11}.
Assume that $A$ and $\Omega$ satisfy Assumption \ref{a1.1}.
\begin{itemize}
\item[\rm(i)] Let $u\in W^{1,2}(\Omega)$ be the weak solution to the Robin problem \eqref{e1.12} with
$\boldsymbol{f}\in\mathrm{BMO}_{r}(\Omega;{\mathbb{R}^n})$. Then $\nabla u\in \mathrm{BMO}_{r}(\Omega;{\mathbb{R}^n})$ and
there exists a positive constant $C$ independent of both $u$ and $\boldsymbol{f}$ such that
\begin{equation*}
\left\|\nabla u\right\|_{\mathrm{BMO}_{r,+}(\Omega;{\mathbb{R}^n})}\le C\left\|\boldsymbol{f}\right\|_{\mathrm{BMO}_{r,+}(\Omega;{\mathbb{R}^n})}.
\end{equation*}
\item[{\rm(ii)}] The Robin problem \eqref{e1.12} with $\boldsymbol{f}\in H^1_{z}(\Omega;{\mathbb{R}^n})$
is uniquely solvable and the weak solution $u\in W^{1,1}(\Omega)$. Furthermore, there exists a positive constant $C$ independent of $u$ and $\boldsymbol{f}$ such that
$$
\|\nabla u\|_{L^1(\Omega;{\mathbb{R}^n})}\le
C\|\boldsymbol{f}\|_{H^1_z(\Omega;{\mathbb{R}^n})}.
$$
\end{itemize}
\end{theorem}

\begin{remark}
We point out that Theorems \ref{th1.1} and \ref{th1.2} and Corollary \ref{c1.1}
also hold for elliptic systems satisfying the strong ellipticity condition
(see, for instance, \cite[(1.2)]{dl21}). This is because the proofs of these results only use the $W^{1,p}$ estimates for elliptic equations, which is also available
for the corresponding elliptic systems (see, for instance, \cite{dk11,dk12,dl21}).
We omit the details in this article.
\end{remark}

The remainder of this article is organized as follows.
In Section \ref{s2}, we present some basic properties of the space $\mathcal{L}^{\sigma(\cdot),p}(\Omega)$,
an equivalent characterization of the space $\mathrm{BMO}_r(\Omega)$, and the pointwise multiplier characterization
of the space $\mathrm{BMO}(\Omega)$. In Section \ref{s3}, some estimates for local solutions to
the second-order elliptic equation \eqref{e1.1} are given. In Section \ref{s4}, we establish the mean oscillation-type
decay estimate of the gradient of solutions to the problem \eqref{e1.6} or \eqref{e1.8} near the boundary of $\Omega$. Finally, the proofs of Theorems \ref{th1.1} and \ref{th1.2} are given, respectively, in Sections \ref{s5} and \ref{s6}.

We finish this section by making some conventions on notation. Throughout the article, we always denote by $C$ or $c$ a \emph{positive constant}, which  
may vary from line
to line. We also use $C_{(\alpha,\beta,\ldots)}$ or $c_{(\alpha,\beta,\ldots)}$ to denote a positive
constant depending on the indicated parameters $\alpha,\beta,\ldots.$ The \emph{symbol} $f\lesssim g$ means that $f\le Cg$. If $f\lesssim g$ and $g\lesssim f$, then we write $f\sim g$.
For each ball $B:=B(x_B,r_B)$ in ${\mathbb{R}^n}$, with $x_B\in{\mathbb{R}^n}$ and $r_B\in (0,\infty)$,
and $\alpha\in(0,\infty)$, let $\alpha B:=B(x_B,\alpha r_B)$.  For any given normed spaces $\mathcal X$
and $\mathcal Y$ with the corresponding norms $\|\cdot\|_{\mathcal X}$ and $\|\cdot\|_{\mathcal Y}$, the
\emph{symbol} ${\mathcal X}\hookrightarrow{\mathcal Y}$ means that, for any $f\in \mathcal X$, $f\in
\mathcal Y$ and $\|f\|_{\mathcal Y}\le C \|f\|_{\mathcal X}$ with the positive constant $C$ independent of $f$.
For any given $n\times n$ matrix ${\mathbf T}$,
denote by ${\mathbf T}^t$ its \emph{transpose matrix}, by ${\mathbf T}^{-1}$
its \emph{inverse matrix} (if the inverse matrix of ${\mathbf T}$ exists), and by
$\mathrm{det}{\mathbf T}$ the \emph{determinant} of ${\mathbf T}$. Furthermore, for any $q\in[1,\infty]$,
we denote by $q'$ its \emph{conjugate exponent}, that is, $1/q+1/q'= 1$. Finally, for any measurable set
$E\subset{\mathbb{R}^n}$ and any (vector-valued or matrix-valued) function $f\in L^1(E)$, we denote the integral
$\int_{E}|f(x)|\,dx$ simply by $\int_{E}|f|\,dx$.

\section{Preliminaries}\label{s2}

In this section, we recall some basic properties of the space $\mathcal{L}^{\sigma(\cdot),p}(\Omega)$,
an equivalent characterization of the space $\mathrm{BMO}_r(\Omega)$, and the pointwise multiplier characterization
of the space $\mathrm{BMO}(\Omega)$.

For the space $\mathcal{L}^{\sigma(\cdot),p}(\Omega)$, we have the following result; see \cite[Theorem 3.1]{n08}.

\begin{lemma}\label{l2.1}
Let $n\ge2$, $\Omega\subset{\mathbb{R}^n}$ be a bounded Lipschitz domain, and $p\in[1,\infty)$.
Assume that the function $\sigma:\ [0,\infty)\to[0,\infty)$ satisfies (ii) and (iii) of
Assumption \ref{a1.1}(a). Then the spaces $\mathcal{L}^{\sigma(\cdot)}(\Omega)=
\mathcal{L}^{\sigma(\cdot),p}(\Omega)$ with equivalent semi-norms.
\end{lemma}

Furthermore, we have the following equivalence of the spaces $\mathrm{BMO}(\Omega)$ and $\mathrm{BMO}_r(\Omega)$.

\begin{lemma}\label{l2.2}
Let $n\ge2$ and $\Omega\subset{\mathbb{R}^n}$ be a bounded Lipschitz domain. Then the spaces $\mathrm{BMO}(\Omega)
=\mathrm{BMO}_r(\Omega)$ with equivalent semi-norms.
\end{lemma}
\begin{proof}
By Jones's extension theorem on the BMO space (see \cite[Theorem 1]{j80}), we know that
the spaces
\begin{equation}\label{e2.1}
\widetilde{\mathrm{BMO}}(\Omega)=\mathrm{BMO}_r(\Omega)
\end{equation}
with equivalent semi-norms.
Here and thereafter, the \emph{space} $\widetilde{\mathrm{BMO}}(\Omega)$ is defined to be the set of all
functions $f\in L^1_\mathrm{loc}(\Omega)$ satisfying
$$\|f\|_{\widetilde{\mathrm{BMO}}(\Omega)}:=\sup_{B\subset\Omega}\fint_{B}|f(x)-(f)_B|\,dx<\infty,$$
where the supremum is taken over all balls $B\subset\Omega$.

From the definitions of $\mathrm{BMO}(\Omega)$ and $\widetilde{\mathrm{BMO}}(\Omega)$,
we deduce that
$$\mathrm{BMO}(\Omega)\hookrightarrow\widetilde{\mathrm{BMO}}(\Omega),$$
which, combined with \eqref{e2.1}, further implies that $\mathrm{BMO}(\Omega)
\hookrightarrow\mathrm{BMO}_r(\Omega)$.

Thus, to finish the proof of the present lemma, it suffices to show that
$\mathrm{BMO}_r(\Omega)\subset\mathrm{BMO}(\Omega)$ and, for any $f\in\mathrm{BMO}_r(\Omega)$,
\begin{equation}\label{e2.2}
\|f\|_{\mathrm{BMO}(\Omega)}\lesssim\|f\|_{\mathrm{BMO}_r(\Omega)}.
\end{equation}
Let $f\in\mathrm{BMO}_r(\Omega)$. Then there exists $\widetilde{f}\in\mathrm{BMO}({\mathbb{R}^n})$ such that
$\widetilde{f}|_\Omega=f$ and $\|\widetilde{f}\|_{\mathrm{BMO}({\mathbb{R}^n})}\sim\|f\|_{\mathrm{BMO}_r(\Omega)}$.

Now, we prove that $f\in\mathrm{BMO}(\Omega)$ and \eqref{e2.2} holds.
Let $B:=B(x_0,r_0)\subset{\mathbb{R}^n}$ with $x_0\in\overline{\Omega}$ and $r_0\in(0,\mathrm{diam\,}(\Omega))$.
By the geometrical property of Lipschitz domains
(see, for instance, \cite[p.\,4]{k94}), we have $|B\cap\Omega|\sim|B|$,
which further implies that
\begin{align}\label{e2.4}
\fint_{B\cap\Omega}|f(y)-(f)_{B\cap\Omega}|\,dy&\lesssim\fint_B\left|\widetilde{f}(y)-\left(\widetilde{f}\right)_{B\cap\Omega}\right|\,dy\\ \nonumber
&\lesssim\fint_B\left|\widetilde{f}(y)-\left(\widetilde{f}\right)_{B}\right|\,dy+\fint_{B\cap\Omega}\left|\widetilde{f}(y)-\left(\widetilde{f}\right)_{B}\right|\,dy\\ \nonumber
&\lesssim\fint_B\left|\widetilde{f}(y)-\left(\widetilde{f}\right)_{B}\right|\,dy\lesssim\left\|\widetilde{f}\right\|_{\mathrm{BMO}({\mathbb{R}^n})}\sim\|f\|_{\mathrm{BMO}_r(\Omega)}.
\end{align}
Therefore, from \eqref{e2.4}, it follows that $f\in\mathrm{BMO}(\Omega)$ and \eqref{e2.2} holds.
This finishes the proof of Lemma \ref{l2.2}.
\end{proof}

Using \cite[Theorem A.8]{cjy16} on the atomic characterization of the Hardy space $H^1_z(\Omega)$,
similarly to the proof of \cite[Theorem 2.1]{C94} (see also \cite[Theorem 4.2]{n08}), we obtain the following
dual results between $H^1_z(\Omega)$ and $\mathrm{BMO}_r(\Omega)$; we omit its proof.

\begin{lemma}\label{l2.3}
Let $n\ge 2$ and $\Omega\subset{\mathbb{R}^n}$ be a bounded Lipschitz domain. Then the dual space of
$H_z^1(\Omega)$ is $\mathrm{BMO}_r(\Omega)$.
\end{lemma}

Let $\Omega\subset{\mathbb{R}^n}$ be a bounded Lipschitz domain. Denote by $M(\mathrm{BMO}(\Omega))$
the space of pointwise multipliers of $\mathrm{BMO}(\Omega)$, namely the space of all functions $g$ such that
$fg\in\mathrm{BMO}(\Omega)$ for any $f\in\mathrm{BMO}(\Omega)$, endowed with the \emph{norm}
$$
\|g\|_{M(\mathrm{BMO}(\Omega))}:=\sup\{\|fg\|_{\mathrm{BMO}(\Omega)}:\ f\in \mathrm{BMO}(\Omega),\
\|f\|_{\mathrm{BMO}(\Omega)}\le1\}.
$$
Then the following equivalent characterization for $M(\mathrm{BMO}(\Omega))$ is well known (see, for instance, \cite{j76,n97,ny97}).

\begin{lemma}\label{l2.4}
Let $n\ge 2$ and $\Omega\subset{\mathbb{R}^n}$ be a bounded Lipschitz domain. Then
$$
M(\mathrm{BMO}(\Omega))=L^\infty(\Omega)\cap \mathcal{L}^{\sigma_0(\cdot)}(\Omega),
$$
where, for any $r\in(0,\mathrm{diam\,}(\Omega))$, $\sigma_0(r):=(1+|\ln r|)^{-1}$.
\end{lemma}

\section{Local solution estimates}\label{s3}
In this section, we establish several estimates for local solutions to second-order elliptic equations
\eqref{e1.1} in the domain $\Omega$.

Let $n\ge2$, $\Omega\subset{\mathbb{R}^n}$ be a bounded Lipschitz domain, and the matrix $A$ be the same as in
\eqref{e1.5}.
A function $u\in W^{1,2}_{\mathrm{loc}}(\Omega)$ is called a \emph{local weak solution} to the equation \eqref{e1.1}
if, for any domain $O$ satisfying $\overline{O}\subset\Omega$ and any $\varphi\in C^\infty_{\rm c}(O)$,
\begin{equation}\label{e3.1}
\int_{O}A(x)\nabla u(x)\cdot\nabla\varphi(x)\,dx=\int_{O}\boldsymbol{f}(x)\cdot\nabla\varphi(x)\,dx
\end{equation}
holds.

Assume that $B\subset\Omega$ is a ball and $u$ is a local weak solution to the equation \eqref{e1.1}.
Then we consider a weak solution $v\in W^{1,2}(B)$ to the Dirichlet problem
\begin{equation}\label{e3.2}
\left\{\begin{array}{ll}
\mathrm{div}(A\nabla v)=0\ \ &\text{in}\ B,\\
v=u\  &\text{on}\  \partial B,
\end{array}\right.
\end{equation}
where the matrix $A$ is the same as in the problem \eqref{e1.6}.
We also point out that, as usual, the boundary condition in \eqref{e3.2} is understood
in the sense that $u-v\in W_{0}^{1,2}(B)$.

For the Dirichlet problem \eqref{e3.2}, we have the following estimate.

\begin{lemma}\label{l3.1}
Let $n\ge2$, $\Omega\subset{\mathbb{R}^n}$ be a bounded Lipschitz domain, and $B\subset\Omega$ be a ball.
Assume that $u$ is a local weak solution to the equation \eqref{e1.1} and $v$ is a weak solution to the
problem \eqref{e3.2}. Then there exists a positive constant $C$, independent of $u,$ $v,$
 $\boldsymbol{f}$, and $B$, such that, for any $\boldsymbol{f}_0\in \mathbb{R}^{n}$,
\begin{equation*}
\fint_{B}\left|\nabla u(x)-\nabla v(x)\right|^2\,dx\le C\fint_{B}\left|\boldsymbol{f}(x)-\boldsymbol{f}_0\right|^2\,dx.
\end{equation*}
\end{lemma}
\begin{proof}
Recall that $v$ is a weak solution to the Dirichlet problem \eqref{e3.2}.
Choosing $u-v\in W^{1,2}_0(B)$ as a test function in \eqref{e3.1} and using the uniform ellipticity condition
\eqref{e1.5},  we have, for any $\boldsymbol{f}_0\in{\mathbb{R}^n}$,
\begin{align*}
\fint_{B}\left|\nabla u(x)-\nabla v(x)\right|^2\,dx&\le\mu_0^{-1}\fint_{B} A(x)\nabla(u-v)(x)\cdot
\nabla(u-v)(x)\,dx\\
&=\mu_0^{-1}\fint_{B} A(x)\nabla u(x)\cdot\nabla(u-v)(x)\,dx\\
&=\mu_0^{-1}\fint_{B}\left[\boldsymbol{f}(x)-\boldsymbol{f}_0\right]\cdot\nabla(u-v)(x)\,dx,
\end{align*}
which, combined with Young's inequality, further implies that, for any given $\delta\in(0,1)$,
\begin{equation}\label{e3.3}
\fint_{B}|\nabla u(x)-\nabla v(x)|^2\,dx\le\delta\fint_{B}|\nabla u(x)-\nabla v(x)|^2\, dx
+C_{(\delta)}\fint_{B}\left|\boldsymbol{f}(x)-\boldsymbol{f}_0\right|^2\,dx.
\end{equation}
Taking $\delta:=\frac12$ in \eqref{e3.3}, we then find that, for any given $\boldsymbol{f}_0\in{\mathbb{R}^n}$,
\begin{equation*}
\fint_{B}|\nabla u(x)-\nabla v(x)|^2\,dx\lesssim\fint_{B}\left|\boldsymbol{f}(x)-\boldsymbol{f}_0\right|^2\,dx.
\end{equation*}
This finishes the proof of Lemma \ref{l3.1}.
\end{proof}

We also have the following mean oscillation estimate for local solutions of \eqref{e1.1}.

\begin{proposition}\label{p3.1}
Let $n\ge2$, $\Omega\subset{\mathbb{R}^n}$ be a bounded Lipschitz domain, and $\delta\in(0,1)$. Assume that
$u$ is a local weak solution to the equation \eqref{e1.1}, where $A$ is a constant matrix and satisfies the uniform ellipticity condition \eqref{e1.5}. Then there exist positive constants $C$ depending only on $n$ and $\mu_0$, but
independent of $\delta$, and $C_{(n,\mu_0,\delta)}$, depending only on $n$, $\mu_0$, and $\delta$, such that,
for any ball $B\subset\Omega$ and $\boldsymbol{f}_0\in \mathbb{R}^{n}$,
\begin{align}\label{e3.4}
\left[\fint_{\delta B}\left|\nabla u(x)-(\nabla u)_{\delta B}\right|^2\,dx\right]^{\frac12}&\le C\delta
\left[\fint_{B}\left|\nabla u(x)-(\nabla u)_{B}\right|^2\,dx\right]^{\frac12}\\ \nonumber
&\quad+C_{(n,\mu_0,\delta)}\left[\fint_{B}\left|\boldsymbol{f}(x)-\boldsymbol{f}_0\right|^2\,dx\right]^{\frac12}.
\end{align}
\end{proposition}
\begin{proof}
Fix a ball $B\subset\Omega$. Let $v$ be a weak solution to the Dirichlet problem \eqref{e3.2}.
From \cite[Lemma 3.10]{hl}, it follows that
$$
\left[\fint_{\delta B}\left|\nabla v(x)-(\nabla v)_{\delta B}\right|^2\,dx\right]^{\frac12}\le C
\delta\left[\fint_{B}\left|\nabla v(x)-(\nabla v)_{B}\right|^2\,dx\right]^{\frac12},
$$
where $C$ is a positive constant independent of $\delta,$ $B$, and $v$.
By this and Lemma \ref{l3.1}, we conclude that, for any $\boldsymbol{f}_0\in{\mathbb{R}^n}$,
\begin{align*}
&\left[\fint_{\delta B}\left|\nabla u(x)-(\nabla u)_{\delta B}\right|^2\,dx\right]^{\frac12}
\le \left[\fint_{\delta B}\left|\nabla u(x)-(\nabla v)_{\delta B}\right|^2\,dx\right]^{\frac12}\\
&\quad\le\left[\fint_{\delta B}\left|\nabla u(x)-\nabla v(x)\right|^2\,dx\right]^{\frac12}+\left[\fint_{\delta B}\left|\nabla v(x)-(\nabla v)_{\delta B}\right|^2\,dx\right]^{\frac12}\\
&\quad\le C\delta^{-n/2}\left[\fint_{B}\left|\boldsymbol{f}(x)-\boldsymbol{f}_0\right|^2\,dx\right]^{\frac12}
+C\delta\left[\fint_{B}|\nabla v(x)-(\nabla v)_{B}|^2\,dx\right]^{\frac12}\\
&\quad\le C\delta^{-n/2}\left[\fint_{B}\left|\boldsymbol{f}(x)-\boldsymbol{f}_0\right|^2\,dx\right]^{\frac12}
+C\delta\left[\fint_{B}|\nabla u(x)-(\nabla u)_{B}|^2\,dx\right]^{\frac12}\\
&\qquad+C\delta\left[\fint_{B}\left|\nabla u(x)-\nabla v(x)\right|^2\,dx\right]^{\frac12}\\
&\quad\le C\delta\left[\fint_{B}\left|\nabla u(x)-(\nabla u)_{B}\right|^2\,dx\right]^{\frac12}
+C\delta^{-n/2}\left[\fint_{B}\left|\boldsymbol{f}(x)-\boldsymbol{f}_0\right|^2\,dx\right]^{\frac12}.
\end{align*}
Thus, \eqref{e3.4} holds. This finishes the proof of Proposition \ref{p3.1}.
\end{proof}

\section{Mean oscillation-type decay estimates near the boundary}\label{s4}

In this section, we establish mean oscillation-type decay estimates of the gradient
of the solution to the problem \eqref{e1.6} or \eqref{e1.8} near the boundary of $\Omega$.
To achieve this, we need to prove a Gehring type estimate for the gradient of solutions to the problem \eqref{e1.6} or \eqref{e1.8} near the boundary of $\Omega$
and to use a flattening technique 
and the pointwise multiplier characterization of the space $\mathrm{BMO}(\Omega)$.

\subsection{A Gehring type estimate near the boundary}

In this subsection, we give a Gehring type estimate for the gradient of solutions to the problem \eqref{e1.6} or \eqref{e1.8} near the boundary of $\Omega$
and the mean oscillation-type decay estimate for the gradient of the solution to the problem
\eqref{e1.6} or \eqref{e1.8} in the interior of $\Omega$.

The following Proposition \ref{p4.1} is known to the expert. However, we give
its proof in this subsection for the sake of completeness.

\begin{proposition}\label{p4.1}
Let $n\ge2$ and $\Omega\subset{\mathbb{R}^n}$ be a bounded Lipschitz domain. Assume that $u$ is a weak solution
to the Neumann problem \eqref{e1.6} with $\boldsymbol{f}\in L^2(\Omega;{\mathbb{R}^n})$. Then there exists
a constant $q_0\in(2,\infty)$, depending on $n,$ $\mu_0$, $A$, and $\Omega$, such that, for any given $q\in(2,q_0)$,
any ball $B:=B(x_0,r_0)\subset{\mathbb{R}^n}$ with $x_0\in\overline{\Omega}$ and $r_0\in(0,\mathrm{diam\,}(\Omega))$, and
any $\boldsymbol{f}_0\in{\mathbb{R}^n}$,
\begin{equation}\label{e4.1}
\fint_{B}|\nabla u(x)|^{q}\,dx\le C\left[\left(\fint_{2B}|\nabla u(x)|\,dx\right)^{q}+\fint_{2B}
|\boldsymbol{f}(x)-\boldsymbol{f}_0|^{q}\,dx\right]
\end{equation}
when $2B\subset\Omega$, and
\begin{equation}\label{e4.2}
\fint_{B\cap\Omega}|\nabla u(x)|^{q}\,dx\le C\left[\left(\fint_{2B\cap\Omega}|\nabla u(x)|\,dx\right)^{q}+\fint_{2B\cap\Omega}
|\boldsymbol{f}(x)|^{q}\,dx\right]
\end{equation}
when $2B\cap\partial\Omega\neq\emptyset$, where $C$ is a positive constant depending only on $n$, $\mu_0$, and $\Omega$, but independent of $B,$ $u$, and
$\boldsymbol{f}$.
\end{proposition}
\begin{proof}
Let $B:=B(x_0,r_0)$ be the same as the present proposition and $\eta\in C^\infty_{\mathrm{c}}({\mathbb{R}^n})$ satisfy that $0\le\eta\le1$,
$\eta\equiv1$ on $B$, $\mathrm{\,supp\,}(\eta)\subset\frac{3}{2}B$, and $|\nabla\eta|\lesssim r_0^{-1}$.
We first assume that $2B\subset\Omega$. In this case, taking $\eta^2[u-(u)_{2B}]$ as a test function,
we obtain that, for any $\boldsymbol{f}_0\in{\mathbb{R}^n}$,
\begin{align*}
\int_\Omega A(x)\nabla u(x)\cdot\nabla\left(\eta^2\left[u-(u)_{2B}\right]\right)(x)\,dx&=
\int_\Omega \boldsymbol{f}(x)\cdot\nabla\left(\eta^2\left[u-(u)_{2B}\right]\right)(x)\,dx\\ \nonumber
&=\int_\Omega[\boldsymbol{f}(x)-\boldsymbol{f}_0]\cdot\nabla\left(\eta^2\left[u-(u)_{2B}\right]\right)(x)\,dx,
\end{align*}
which further implies that
\begin{align*}
\int_\Omega \left[\eta(x)\right]^2A(x)\nabla u(x)\cdot\nabla u(x)\,dx&=-\int_\Omega 2\eta(x)\left[u(x)-(u)_{2B}\right]
A(x)\nabla u(x)\cdot\nabla \eta(x)\,dx\\ \nonumber
&\quad+\int_\Omega \left[\eta(x)\right]^2[\boldsymbol{f}(x)-\boldsymbol{f}_0]\cdot\nabla u(x)\,dx\\ \nonumber
&\quad+\int_\Omega 2\eta(x)\left[u(x)-(u)_{2B}\right][\boldsymbol{f}(x)-\boldsymbol{f}_0]\cdot\nabla \eta(x)\,dx.
\end{align*}
From this, the uniform ellipticity condition \eqref{e1.5}, and Young's inequality,
it follows that, for any given $\varepsilon\in(0,1)$,
\begin{align}\label{e4.3}
\fint_{2B} |\nabla u(x)|^2[\eta(x)]^2\,dx&\le\mu_0^{-1}\fint_{2B}[\eta(x)]^2
A(x)\nabla u(x)\cdot\nabla u(x)\,dx\\ \nonumber
&\le\varepsilon\fint_{2B} |\nabla u(x)|^2[\eta(x)]^2\,dx+C_{(\varepsilon)}\fint_{2B}
\left|\boldsymbol{f}(x)-\boldsymbol{f}_0\right|^2[\eta(x)]^2\,dx\\ \nonumber
&\quad+C_{(\varepsilon)}r_0^{-2}\fint_{2B}|u(x)-(u)_{2B}|^2\,dx.
\end{align}
Take $\varepsilon:=1/2$ in \eqref{e4.3}. Then, by \eqref{e4.3}, the assumption that $\eta\equiv1$ on
$B$, and the Sobolev--Poincar\'e inequality (see, for instance, \cite[Theorem 1.1]{bk95}), we further deduce that
\begin{align}\label{e4.4}
\left[\fint_{B} |\nabla u(x)|^2\,dx\right]^{\frac12}&\lesssim\left[\fint_{2B}|\nabla u(x)|^{\frac{2n}{n+2}}\,dx\right]^{\frac{n+2}{2n}}
+\left[\fint_{2B}\left|\boldsymbol{f}(x)-\boldsymbol{f}_0\right|^2\,dx\right]^{\frac12}.
\end{align}
Using \eqref{e4.4} and a version of Gehring's lemma as in \cite{i98}, we conclude that there
exists an exponent $q_0\in(2,\infty)$ such that, for any given $q\in(2,q_0)$, the estimate \eqref{e4.1}
holds.

Next, assume that $2B\cap\partial\Omega\neq\emptyset$. In this case,
taking $\eta^2[u-(u)_{2B\cap\Omega}]$ as a test function, we then have
\begin{align}\label{e4.5}
\int_\Omega A(x)\nabla u(x)\cdot\nabla\left(\eta^2\left[u-(u)_{2B\cap\Omega}\right]\right)(x)\,dx=
\int_\Omega \boldsymbol{f}(x)\cdot\nabla\left(\eta^2\left[u-(u)_{2B\cap\Omega}\right]\right)(x)\,dx.
\end{align}
Using \eqref{e4.5} and repeating the proof of \eqref{e4.4}, we conclude that
\eqref{e4.2} also holds in this case. This finishes the proof of Proposition \ref{p4.1}. 	
\end{proof}

\begin{proposition}\label{p4.2}
Let $n\ge2$ and $\Omega\subset{\mathbb{R}^n}$ be a bounded Lipschitz domain. Assume that $u$ is the weak solution
to the Dirichlet problem \eqref{e1.8} with $\boldsymbol{f}\in L^2(\Omega;{\mathbb{R}^n})$. Then there exists
a constant $q_0\in(2,\infty)$, depending on $n$, $\mu_0,$ $A$, and $\Omega$, such that, for any given $q\in(2,q_0)$,
any ball $B:=B(x_0,r_0)\subset{\mathbb{R}^n}$ with $x_0\in\overline{\Omega}$ and $r_0\in(0,\mathrm{diam\,}(\Omega))$, and
any $\boldsymbol{f}_0\in{\mathbb{R}^n}$,
\begin{equation}\label{e4.6}
\fint_{B\cap\Omega}|\nabla u(x)|^{q}\,dx\le C\left[\left(\fint_{2B\cap\Omega}|\nabla u(x)|\,dx\right)^{q}
+\fint_{2B\cap\Omega}\left|\boldsymbol{f}(x)-\boldsymbol{f}_0\right|^{q}\,dx\right],
\end{equation}
where $C$ is a positive constant depending only on $n$, $\mu_0$, and $\Omega$, but independent of $B,$ $u$, and $\boldsymbol{f}$.
\end{proposition}

\begin{proof}
Since $u$ is the weak solution of the Dirichlet problem \eqref{e1.8},
it follows that, for any $\varphi\in C^\infty_{\rm c}(\Omega)$ and any $\boldsymbol{f}_0\in{\mathbb{R}^n}$, \eqref{e1.9} holds.
Using \eqref{e1.9}, similarly to the proof of Proposition \ref{p4.1}, we obtain \eqref{e4.6};
we omit the details.
\end{proof}

We also have the following interior mean oscillation-type decay estimate for the weak solution to the problems
\eqref{e1.6} and \eqref{e1.8}.

\begin{proposition}\label{p4.3}
Let $n\ge2$ and $\Omega\subset{\mathbb{R}^n}$ be a bounded Lipschitz domain. Assume that the matrix $A$
and the function $\sigma$ satisfy Assumption \ref{a1.1}(a). Assume further that $\theta\in(0,1)$, $\boldsymbol{f}\in
L^2(\Omega;{\mathbb{R}^n})$, and $u\in W^{1,2}(\Omega)$ is the weak solution to the Neumann problem \eqref{e1.6} or the Dirichlet
problem \eqref{e1.8}. Then there exists a constant $q\in(2,\infty)$ such that, for any ball $B:=B(x,r)$
satisfying $2B\subset\Omega$ and $r\in(0,R_0)$ with $R_0\in(0,\infty)$ being the same as in Assumption \ref{a1.1}(a)
and for any $\boldsymbol{f}_0\in{\mathbb{R}^n}$,
\begin{align}\label{e4.7}
&\left[\fint_{\theta B}|\nabla u(y)-(\nabla u)_{\theta B}|^2\,dy\right]^{\frac{1}{2}}\\ \notag
&\quad\le C_{(n,\mu_0,\theta)}\sigma(r)\fint_{2B}
|\nabla u(y)|\,dy+C_{(n,\mu_0,\theta,R_0)}\left[\fint_{2B}\left|\boldsymbol{f}(y)-\boldsymbol{f}_0
\right|^{q}\,dy\right]^{\frac{1}{q}}\\ \notag
&\qquad+C\theta\left[\fint_{2B}|\nabla u(y)-(\nabla u)_{2B}|^2\,dy\right]^{\frac{1}{2}},
\end{align}
where $C_{(n,\mu_0,\theta)}$ is a positive constant depending on $n$, $\mu_0$, and $\theta$, $C_{(n,\mu_0,\theta,R_0)}$
is a positive constant depending on $n$, $\mu_0$, $\theta$, and $R_0$,
and $C$ is a positive constant independent of $u,$ $\boldsymbol{f},$ $B$, and $\theta$.
\end{proposition}

\begin{proof}
Let   
$A_0:=(A)_{B}$. Since $u$ is the weak
solution to the Neumann problem \eqref{e1.6} or the Dirichlet problem \eqref{e1.8}, it follows that
\begin{equation}\label{e4.8}
\mathrm{div}{\left(A_0\nabla u\right)}=\mathrm{div}{\left((A_0-A)\nabla u+\boldsymbol{f}\right)}\ \text{in}\ \Omega.
\end{equation}
Then, by \eqref{e4.8} and \eqref{e3.4}, there exist a positive constant $C$,
independent of $\theta,$ $B,$ $u$, and $\boldsymbol{f}$, and a positive constant $C_{(n,\mu_0,\theta)}$, depending only
on $n$, $\mu_0$, and $\theta$, such that, for any $\boldsymbol{f}_0\in{\mathbb{R}^n}$,
\begin{align}\label{e4.9}
\left[\fint_{\theta B}|\nabla{u}(y)-(\nabla u)_{\theta B}|^2\,dy\right]^{\frac{1}{2}}&\le C\theta\left[\fint_{B}|\nabla
{u}(y)-(\nabla u)_{B}|^2\,dy\right]^{\frac{1}{2}}\\ \notag
&\quad+C_{(n,\mu_0,\theta)}\left[\fint_{B}\left|A_0\nabla{u}(y)-A(y)\nabla{u}(y)\right|^2\,dy\right]^{\frac{1}{2}}\\ \notag
&\quad+C_{(n,\mu_0,\theta)}\left[\fint_{B}\left|\boldsymbol{f}(y)-\boldsymbol{f}_0\right|^2\,dy\right]^{\frac{1}{2}}.
\end{align}
Moreover, from \eqref{e4.1}, Assumption \ref{a1.1}$\mathrm{(a)}$, Lemma \ref{l2.1},
and H\"older's inequality, we deduce that there exist constants $q\in(2,\infty)$ and $C\in(0,\infty)$
such that
\begin{align*}
&\left[\fint_{B}\left|A_0\nabla{u}(y)-A(y)\nabla{u}(y)\right|^2\,dy\right]^{\frac{1}{2}}\\ \notag
&\quad\le \left[\fint_{B}\left|A_0-A(y)\right|^{2(\frac{q}{2})'}\,dy\right]^{\frac{1}{2(\frac{q}{2})'}}
\left[\fint_{B}|\nabla{u}(y)|^q\,dy\right]^{\frac{1}{q}}\\ \notag
&\quad\le C\sigma(r)\left\{\fint_{2B}|\nabla u(y)|\,dy+\left[\fint_{2B}\left|\boldsymbol{f}(y)
-\boldsymbol{f}_0\right|^q\,dy\right]^{\frac{1}{q}}\right\},
\end{align*}
which, together with \eqref{e4.9} and Assumption \ref{a1.1}$\mathrm{(a)}$, further implies that \eqref{e4.7} holds.
This finishes the proof of Proposition \ref{p4.3}.
\end{proof}

\subsection{Change of coordinates}\label{s4.2}

In this subsection, we recall some necessary results on the change of coordinates for Lipschitz
domains (see, for instance, \cite[Section 4.2]{bc22}, \cite[Section 5.1]{b05}, and \cite[p.\,50]{dk10}).

Let $n\ge2$ and $\Omega\subset{\mathbb{R}^n}$ be a bounded Lipschitz domain. By the definition of Lipschitz domains,
there exists a constant $R_1\in(0,\infty)$, depending only on $\Omega$, such that, for any $x\in\partial\Omega$
and $r\in(0,R_1]$, $B(x,r)\cap\partial\Omega$ is a part of some Lipschitz graph.
Without loss of generality, we may assume that the origin $\mathbf{0}\in\partial\Omega$ and there exists a Lipschitz map $\psi:\ \mathbb{R}^{n-1}\to\mathbb{R}$ such that
\begin{equation}\label{e4.10}
\partial\Omega\cap B(\mathbf{0},R_1)=\left\{(x',\psi(x'))\in{\mathbb{R}^n}:\ (x',0)\in B(\mathbf{0},R_1)\right\}
\end{equation}
and
\begin{equation*}
\Omega\cap B(\mathbf{0},R_1)=\left\{(x',x_n)\in B(\mathbf{0},R_1):\ x_n>\psi(x')\right\}.
\end{equation*}
Then the map $\Psi:\ \Omega\cap B(\mathbf{0},R_1)\to B^+(\mathbf{0},R_1)$ is defined by setting,
for any $x:=(x',x_n)\in\Omega\cap B(\mathbf{0},R_1)$,
\begin{equation*}
\Psi(x):=(x',x_n-\psi(x')).
\end{equation*}
Here and thereafter, $B^+(\mathbf{0},R_1):=\{y:=(y',y_n)\in B(\mathbf{0},R_1):\ y_n>0\}$.
It is easy to see that
$$\Psi\left(\partial\Omega\cap B(\mathbf{0},R_1)\right)\subset\left\{(y',y_n)\in{\mathbb{R}^n}:\ y_n=0\right\}$$
and $\Psi({\bf 0})={\bf 0}$. The function $\Psi:\ \overline{\Omega}\cap B(\mathbf{0},R_1)
\to\Psi(\overline{\Omega}\cap B(\mathbf{0},R_1))$ is invertible, with a Lipschitz
continuous inverse
$$\Psi^{-1}:\ \Psi\left(\overline{\Omega}\cap B(\mathbf{0},R_1)\right)\to\overline{\Omega}\cap B(\mathbf{0},R_1).$$
Furthermore, the map $\mathbf{J}:\ \overline{\Omega}\cap B(\mathbf{0},R_1)\to\mathbb{R}^{n\times n}$
is defined by setting, for any $x\in\overline{\Omega}\cap B(\mathbf{0},R_1)$,
\begin{equation}\label{e4.11}
\mathbf{J}(x):=\nabla\Psi(x).
\end{equation}
Thus, for any $x:=(x',x_n)\in\overline{\Omega}\cap B(\mathbf{0},R_1)$,
\begin{equation*}
\mathbf{J}(x)=
\begin{pmatrix}\mathbf{I}_{n-1}
  & 0\\-\nabla\psi(x')
  &1
\end{pmatrix}=\begin{pmatrix}
  1&  0&  0& \cdots &0 \\
  0& 1& 0& \cdots  &  0\\
  \vdots &\vdots  &\ddots &\ddots& \vdots\\
  0 &0&\cdots  & 1 & 0\\
  -\psi_{x_1}(x')&-\psi_{x_2}(x')& \cdots & -\psi_{x_{n-1}}(x') &1
\end{pmatrix}.
\end{equation*}
Here and thereafter, $\mathbf{I}_{n-1}$ denotes the $(n-1)\times(n-1)$ unit matrix. Next we define the map
$$\mathbf{J}^{-1}:\ \Psi\left(\overline{\Omega}\cap B(\mathbf{0},R_1)\right)\to\mathbb{R}^{n\times n}$$
by setting, for any $y\in\Psi(\overline{\Omega}\cap B(\mathbf{0},R_1))$,
$$\mathbf{J}^{-1}(y):=\nabla(\Psi^{-1}(y)).$$
Therefore, for any $y\in\Psi(\overline{\Omega}\cap B(\mathbf{0},R_1))$,
$\mathbf{J}^{-1}(y)=(\nabla\Psi)^{-1}(\Psi^{-1}(y)).$
\begin{remark}\label{r4.1}
By the definitions of $\mathbf{J}$ and $\mathbf{J}^{-1}$, we have
\begin{enumerate}
\item[\rm(i)] For any $y\in\Psi(\overline{\Omega}\cap B(\mathbf{0},R_1))$,
$\mathbf{J}^{-1}(y)\mathbf{J}(\Psi^{-1}(y))=\mathbf{I}_n$, where $\mathbf{I}_n$ denotes
the $n\times n$ unit matrix.
\item[\rm(ii)] For any $x\in\overline{\Omega}\cap B(\mathbf{0},R_1)$, $\det{\mathbf{J}(x)}=1$ and,
for any $y\in\Psi(\overline{\Omega}\cap B(\mathbf{0},R_1))$, $\det{\mathbf{J}^{-1}(y)}=1$.
\item[\rm(iii)] $|\Psi(E)|=|E|$ for any measurable set $E\subset\overline{\Omega}\cap B(\mathbf{0},R_1)$,
and $|\Psi^{-1}(E)|=|E|$ for any measurable set $E\subset\Psi(\overline{\Omega}\cap B(\mathbf{0},R_1))$.
\end{enumerate}
\end{remark}

Furthermore, since both $\Psi$ and $\Psi^{-1}$ are Lipschitz continuous, it follows that there
exist constants
\begin{equation}\label{e4.12}
0<\lambda\le 1\le\Lambda
\end{equation}
such that, for any $r\in(0,R_1]$,
\begin{equation}\label{e4.13}
B^+(\mathbf{0},\lambda r)\subset\Psi(\Omega\cap B(\mathbf{0},r))\subset B^+(\mathbf{0},\Lambda r)
\end{equation}
and, for any $r\in(0,\infty)$ satisfying $B^+(\mathbf{0},r)\subset\Psi(\overline{\Omega}\cap B(\mathbf{0},R_1))$,
\begin{equation}\label{e4.14}
\Omega\cap B\left(\mathbf{0},\Lambda^{-1}r\right)\subset\Psi^{-1}(B(\mathbf{0},r))\subset\Omega\cap
B\left(\mathbf{0},\lambda^{-1}r\right).
\end{equation}

For any given (vector-valued or matrix-valued) function $f$ on $\overline{\Omega}\cap B(\mathbf{0},R_1)$, define the function
$\widetilde{f}$ on $\Psi(\overline{\Omega}\cap B(\mathbf{0},R_1))$ by setting, for any
$y\in\Psi(\overline{\Omega}\cap B(\mathbf{0},R_1))$,
\begin{equation}\label{e4.15}
\widetilde{f}(y):=f\left(\Psi^{-1}(y)\right).
\end{equation}
If $f$ is differentiable, then, for any $y\in\Psi(\overline{\Omega}\cap B(\mathbf{0},R_1))$,
$$\nabla_{y}\widetilde{f}(y)=\nabla_{x}f\left(\Psi^{-1}(y)\right)\mathbf{J}^{-1}(y)$$
and, for any $x\in\overline{\Omega}\cap B(\mathbf{0},R_1)$,
$$\nabla_{x}f(x)=\nabla_{y}\widetilde{f}\left(\Psi(x)\right)\mathbf{J}(x),$$
where $\nabla_x$ and $\nabla_y$, respectively, denote the gradient with respect to the variables $x$
and $y$. By the boundedness of $\mathbf{J}$, we conclude that, if $y=\Psi(x)$, then
$$|\nabla_xf(x)|\sim\left|\nabla_y\widetilde{f}(y)\right|$$
with the positive equivalence constants depending only on the Lipschitz constants of $\Psi$ and $\Psi^{-1}$.

Furthermore, recall that the function $u$ is called a \emph{weak solution} to the Neumann problem
\begin{equation}\label{e4.16}
\left\{\begin{array}{ll}
\mathrm{div}(A\nabla u)=\mathrm{div} \boldsymbol{f}\ \ &\text{in}\ \Omega\cap B(\mathbf{0},R_1),\\
\dfrac{\partial u}{\partial\boldsymbol{\nu}}=\boldsymbol{f}\cdot\boldsymbol{\nu}\  &\text{on}\
\partial\Omega\cap B(\mathbf{0},R_1)
\end{array}\right.
\end{equation}
if
$$\int_{\Omega\cap B(\mathbf{0},R_1)}A(x)\nabla u(x)\cdot\nabla\varphi(x)\,dx=\int_{\Omega\cap B(\mathbf{0},R_1)}
\boldsymbol{f}(x)\cdot\nabla\varphi(x)\,dx
$$
holds for any $\varphi\in C^\infty(\Omega\cap B(\mathbf{0},R_1))$ with $\varphi=0$ on $\Omega\cap\partial B(\mathbf{0},R_1)$.

Now, we show that, if $u$ is a weak solution of \eqref{e4.16}, then the function $\widetilde{u}$,
defined via replacing $f$ by $u$ in \eqref{e4.15}, is a weak solution of
a similar Neumann problem. Indeed, since $\det{\mathbf{J}^{-1}}=1$, it follows that,
for any $\varphi\in C^{\infty}(\Omega\cap B(\mathbf{0},R_1))$ with $\varphi=0$ on $\Omega\cap\partial B(\mathbf{0},R_1)$,
\begin{align}\label{e4.17}
&\int_{\Omega\cap B(\mathbf{0},R_1)}A(x)\nabla_xu(x)\cdot\nabla_x\varphi(x)\,dx\\ \nonumber
&\quad=\int_{\Psi(\Omega\cap B(\mathbf{0},R_1))}A\left(\Psi^{-1}(y)\right)\nabla_xu\left(\Psi^{-1}(y)\right)
\cdot\nabla_x\varphi\left(\Psi^{-1}(y)\right)\det{\mathbf{J}^{-1}(y)}\,dy\\ \nonumber
&\quad=\int_{\Psi(\Omega\cap B(\mathbf{0},R_1))}A\left(\Psi^{-1}(y)\right)\nabla_y\widetilde{u}(y)\mathbf{J}\left(\Psi^{-1}(y)\right)\cdot
\nabla_y\widetilde{\varphi}(y)\mathbf{J}\left(\Psi^{-1}(y)\right)\,dy\\ \nonumber
&\quad=\int_{\Psi(\Omega\cap B(\mathbf{0},R_1))}\mathbf{J}\left(\Psi^{-1}(y)\right)A\left(\Psi^{-1}(y)\right)
\mathbf{J}^{t}\left(\Psi^{-1}(y)\right)\nabla_y\widetilde{u}(y)
\cdot\nabla_y\widetilde{\varphi}(y)\,dy,
\end{align}
where $\widetilde{\varphi}$ is defined by replacing $f$
with $\varphi$ in \eqref{e4.15}. Similarly to \eqref{e4.17}, we also obtain
$$\int_{\Omega\cap B(\mathbf{0},R_1)} \boldsymbol{f}(x)\cdot\nabla_x\varphi(x)\,dx
=\int_{\Psi(\Omega\cap B(\mathbf{0},R_1))}\widetilde{\boldsymbol{f}}(y)\mathbf{J}^{t}\left(\Psi^{-1}(y)\right)
\cdot\nabla_y\widetilde{\varphi}(y)\,dy,$$
where $\widetilde{\boldsymbol{f}}$ is defined by replacing $f$ with $\boldsymbol{f}$
in \eqref{e4.15}, which, combined with \eqref{e4.17}, further implies that
\begin{align}\label{e4.18}
&\int_{\Psi(\Omega\cap B(\mathbf{0},R_1))}\mathbf{J}\left(\Psi^{-1}(y)\right)A\left(\Psi^{-1}(y)\right)
\mathbf{J}^{t}\left(\Psi^{-1}(y)\right)\nabla_y\widetilde{u}(y)\cdot\nabla_y\widetilde{\varphi}(y)\,dy\\ \nonumber
&\quad=\int_{\Psi(\Omega\cap B(\mathbf{0},R_1))}\widetilde{\boldsymbol{f}}(y)\mathbf{J}^{t}\left(\Psi^{-1}(y)\right)
\cdot\nabla_y\widetilde{\varphi}(y)\,dy.
\end{align}
Therefore, by \eqref{e4.18}, we have the following lemma.

\begin{lemma}\label{l4.1}
Let $R_1\in(0,\infty)$ be the same as in \eqref{e4.10} and let $u$ be a weak solution of \eqref{e4.16}.
Then $\widetilde{u}$ is a weak solution of the Neumann problem
\begin{equation*}
\left\{\begin{array}{ll}
\mathrm{div}_y\left(A_{\widetilde{\mathbf{J}}}\nabla_y\widetilde{u}\right)=\mathrm{div}\left(\widetilde{\boldsymbol{f}}\widetilde{\mathbf{J}^t}\right)\ \
&\text{in}\ B^+({\bf 0},{\lambda R_1}),\\
\dfrac{\partial \widetilde{u}}{\partial\boldsymbol{\nu}}=\widetilde{\boldsymbol{f}}\widetilde{\mathbf{J}^t}
\cdot\boldsymbol{\nu}\  &\text{on}\  \left\{y_n=0\right\}\cap B({\bf 0},{\lambda R_1}),
\end{array}\right.
\end{equation*}
where, for any $y\in B^+({\bf 0},{\lambda R_1})$, $A_{\widetilde{\mathbf{J}}}(y):
=\mathbf{J}(\Psi^{-1}(y))A(\Psi^{-1}(y))\mathbf{J}^{t}(\Psi^{-1}(y))$
and $\widetilde{\mathbf{J}^t}(y):=\mathbf{J}^t(\Psi^{-1}(y))$.
\end{lemma}

For any given $s\in(0,\lambda R_1)$, define the matrix $\mathbf{J}_s\in\mathbb{R}^{n\times n}$ by
setting
\begin{equation}\label{e4.19}
\mathbf{J}_s:=(\mathbf{J})_{\Omega\cap B({\bf0},s)}.
\end{equation}
Based on \eqref{e4.13} and \eqref{e4.14}, we choose the constant $\Lambda$ in \eqref{e4.12} large enough
such that $\mathbf{J}_sB({\bf0},s)\subset B({\bf0},\Lambda R_1)$ for any given $s\in(0,\lambda R_1)$.
For any constant-valued matrix ${\bf T}\in\mathbb{R}^{n\times n}$, let
$$H_{\bf T}:=\{x\in{\mathbb{R}^n}:\ ({\bf T}x)_{n}>0\}.$$
For any given function $u$ on $\Omega$, define $\overline{{u}}:\ H_{{\bf J}_s}\cap
B({\bf0},\frac{\lambda}{\Lambda}R_1)\to\mathbb{R}$ by setting, for any
$z\in H_{\mathbf{J}_s}\cap B({\bf0},\frac{\lambda}{\Lambda}R_1),$
\begin{equation}\label{e4.20}
\overline{{u}}(z):=\widetilde{u}\left(\mathbf{J}_sz\right).
\end{equation}
Therefore, for any $z\in H_{\mathbf{J}_s}\cap B({\bf0},\frac{\lambda}{\Lambda}R_1)$,
$$\nabla_z\overline{{u}}(z)=\nabla_y\widetilde{u}\left({\bf J}_sz\right){\bf J}_s.$$
By the definitions of $\mathbf{J}$ and $\mathbf{J}_s$, for any given
$z:=(z',z_n)\in{\mathbb{R}^n}$,
$$(\mathbf{J}_sz)_n\ge 0\ \text{if and only if}\ z_n\ge(\nabla\psi)_{\Omega\cap B({\bf0},s)}\cdot z'.$$
Furthermore, for any vector-valued (or matrix-valued) function $\boldsymbol{f}$ on $\Omega\cap B({\bf0}, R_1)$,
define the function $\overline{{\boldsymbol{f}}}:\ H_{\mathbf{J}_s}\cap B({\bf0},{\frac{\lambda}{\Lambda}R_1})
\to\mathbb{R}^n$ by setting, for any $z\in H_{\mathbf{J}_s}\cap B({\bf0},{\frac{\lambda}{\Lambda}R_1})$,
\begin{equation}\label{e4.21}
\overline{{\boldsymbol{f}}}(z):=\widetilde{\boldsymbol{f}}\left(\mathbf{J}_sz\right).
\end{equation}
Then, as in Lemma \ref{l4.1}, we obtain the following result.

\begin{lemma}\label{l4.2}
Let $R_1\in(0,\infty)$ be the same as in \eqref{e4.10} and $u$ be a weak solution of \eqref{e4.16}.
Then $\overline{u}$, defined as in \eqref{e4.20}, is a weak solution of the Neumann problem
\begin{equation}\label{e4.22}
\left\{\begin{array}{ll}
\mathrm{div}_z\left(A_{{\bf J}_s^{-1}\overline{{\bf J}}}\nabla_z\overline{{u}}\right)
=\mathrm{div}\left(\overline{{\boldsymbol{f}}}\,\overline{{\bf J}^t}\left({\bf J}_s^{-1}\right)^t\right)\ \
&\text{in}\ H_{\mathbf{J}_s}\cap B\left({\bf0},{\frac{\lambda}{\Lambda}R_1}\right),\\
\dfrac{\partial \overline{{u}}}{\partial\boldsymbol{\nu}}=\overline{{\boldsymbol{f}}}
\,\overline{{\bf J}^t}\left({\bf J}_s^{-1}\right)^t\cdot\boldsymbol{\nu}\
&\text{on}\  \left\{(\mathbf{J}_sz)_n=0\right\}\cap B\left({\bf0},{\frac{\lambda}{\Lambda}R_1}\right),
\end{array}\right.
\end{equation}
where, for any $z\in H_{\mathbf{J}_s}\cap B({\bf0},{\frac{\lambda}{\Lambda}R_1})$,
$\overline{{\bf J}}(z)$ is defined as in \eqref{e4.21} with ${\bf J}$ in place of $\boldsymbol{f}$ and $A_{{\bf J}_s^{-1}\overline{{\bf J}}}(z):
={\bf J}_s^{-1}\overline{{\bf J}}(z)\overline{{A}}(z)\overline{{\bf J}^t}(z)
({\bf J}_s^{-1})^t$.
\end{lemma}

Let $A_0\in\mathbb{R}^{n\times n}$ be a constant-valued matrix satisfying the uniform ellipticity condition
\eqref{e1.5}. Assume that $s\in(0,R_1)$, $\boldsymbol{F}\in L^2(B({\bf0},R_1);{\mathbb{R}^n})$,
and $w$ is a weak solution of the Neumann problem
\begin{equation}\label{e4.23}
\left\{\begin{array}{ll}
\mathrm{div}_z\left(A_0\nabla_zw\right)=\mathrm{div}\left(\boldsymbol{F}\right)\ \ &\text{in}\ H_{\mathbf{J}_s}
\cap B\left({\bf0},{\frac{\lambda}{\Lambda}R_1}\right),\\
\dfrac{\partial w}{\partial\boldsymbol{\nu}}=\boldsymbol{F}\cdot\boldsymbol{\nu}\  &\text{on}\
\left\{(\mathbf{J}_sz)_n=0\right\}\cap B\left({\bf0},{\frac{\lambda}{\Lambda}R_1}\right).
\end{array}\right.
\end{equation}
Then, by an argument similar to that used in the proof of \cite[Proposition 3.2]{dlk20} and a change of variables,
we obtain the following lemma for the Neumann problem \eqref{e4.23}; we omit its proof.

\begin{lemma}\label{l4.3}
Let $R_1\in(0,\infty)$ be the same as in \eqref{e4.10} and $w$ be a weak solution of \eqref{e4.23}.
Then, for any given $s\in(0,\lambda R_1)$, any $r\in(0,\lambda R_1/\Lambda)$, $\delta\in(0,1)$, and $\boldsymbol{F}_0\in{\mathbb{R}^n}$,
\begin{align*}
&\left[\fint_{ B({\bf0},\delta r)\cap H_{{\bf J}_s}}\left|\nabla w(z)-(\nabla w)_{D_{\delta r}}\right|^2\,dz\right]^{\frac{1}{2}}\\ \nonumber
&\quad\le C\delta\left[\fint_{B({\bf0},r)\cap H_{{\bf J}_s}}\left|\nabla w(z)-(\nabla w)_{D_{r}}\right|^2\,dz\right]^{\frac{1}{2}}
+C_{(n,\mu_0,\delta)}\left[\fint_{B({\bf0},r)\cap H_{{\bf J}_s}}\left|\boldsymbol{F}(z)-\boldsymbol{F}_0\right|^2\,dz\right]^{\frac{1}{2}},
\end{align*}
where $D_r:=B({\bf0},r)\cap H_{{\bf J}_s}$, $(\nabla w)_{D_{r}}:=((w_{z_1})_{D_r},\ldots,(w_{z_{n}})_{D_r})$, $C$ is
a positive constant independent of $\delta,$ $s,$ $r,$ $w$, and $\boldsymbol{F}$, and $C_{(n,\mu_0,\delta)}$
is a positive constant depending only on $\delta$, $\mu_0$, and $n$.
\end{lemma}

Let $A_0\in\mathbb{R}^{n\times n}$, $s\in(0,R_1)$, and $\boldsymbol{F}\in L^2(B({\bf0},R_1);{\mathbb{R}^n})$ be the same as in \eqref{e4.23}.
Assume that $v$ is a weak solution of the Dirichlet problem
\begin{equation}\label{e4.24}
\left\{\begin{array}{ll}
\mathrm{div}_z\left(A_0\nabla_zv\right)=\mathrm{div}\left(\boldsymbol{F}\right)\ \ &\text{in}\ H_{\mathbf{J}_s}\cap
B\left({\bf0},{\frac{\lambda}{\Lambda}R_1}\right),\\
v=0\  &\text{on}\ \left\{(\mathbf{J}_sz)_n=0\right\}\cap B\left({\bf0},{\frac{\lambda}{\Lambda}R_1}\right).
\end{array}\right.
\end{equation}
By applying an argument similar to that used in the proof of
\cite[Lemma 2.8]{dek18} and a change of variables, we also obtain the following lemma for the Dirichlet
problem \eqref{e4.24}; we omit its proof.

\begin{lemma}\label{l4.4}
Let $R_1\in(0,\infty)$ be the same as in \eqref{e4.10} and let $v$ be a weak solution of the Dirichlet
problem \eqref{e4.24}. Then, for any given $s\in(0,\lambda R_1)$, any $r\in(0,\lambda R_1/\Lambda)$,
$\delta\in(0,1)$, and $\boldsymbol{F}_0\in{\mathbb{R}^n}$,
\begin{align*}
&\left[\fint_{\delta B({\bf0},r)\cap H_{{\bf J}_s}}\left|\nabla v(z)-\boldsymbol{V}_{\delta r}{\bf J}_s\right|^2\,dz\right]^{\frac{1}{2}}\\ \nonumber
&\quad\le C\delta\left[\fint_{B({\bf0},r)\cap H_{{\bf J}_s}}\left|\nabla v(z)-\boldsymbol{V}_r{\bf J}_s\right|^2\,dz\right]^{\frac{1}{2}}
+C_{(n,\mu_0,\delta)}\left[\fint_{B({\bf0},r)\cap H_{{\bf J}_s}}|\boldsymbol{F}(z)-\boldsymbol{F}_0|^2\,dz\right]^{\frac{1}{2}},
\end{align*}
where $\boldsymbol{V}_r:=(0,\ldots,0,(v_{z_{n}})_{D_r})$ with $D_r:=B({\bf0},r)\cap
H_{{\bf J}_s}$, ${\bf J}_s$ is the same as in \eqref{e4.19}, $C$ is a positive constant independent of $\delta,$ $s,$ $r,$ $v$,
and $\boldsymbol{F}$, and $C_{(n,\mu_0,\delta)}$ is a positive constant depending only on $\delta$, $\mu_0$, and $n$.
\end{lemma}

\subsection{Decay estimate near the boundary}

In this subsection, we prove the mean oscillation-type decay estimate of the gradient
of the solution to the problem \eqref{e1.6} or \eqref{e1.8} near the boundary of $\Omega$,
which plays a key role in the proof of Theorem \ref{th1.1}.

\begin{theorem}\label{t4.1}
Let $n\ge2$, $\Omega$ be a bounded Lipschitz domain of ${\mathbb{R}^n}$, and $x\in\partial\Omega$, and let $R_2:=\min\{R_0,R_1\}$,
where $R_0$ is the same as in Assumption \ref{a1.1} and $R_1$ the same as in \eqref{e4.10}.
Assume that the matrix $A$ and the domain $\Omega$ satisfy Assumption \ref{a1.1}. Assume further that $\theta\in(0,1)$,
$\boldsymbol{f}\in \mathrm{BMO}_r(\Omega;{\mathbb{R}^n})$, and $u$ is the weak solution to the Neumann problem
\eqref{e1.6}. Then there exist constants $C\in(0,\infty)$ and $q\in(2,\infty)$, independent of $\theta$,
such that, for any given $s\in(0,R_2]$,
\begin{align}\label{e4.25}
&\left[\fint_{\Omega\cap B(x,\theta s)}\left|\nabla u(y)-(\nabla u)_{\Omega\cap B(x,\theta s)}\right|^2\,dy\right]^{\frac{1}{2}}\\ \notag
&\quad\le C\theta\left[\fint_{\Omega\cap B(x,s)}\left|\nabla u(y)-(\nabla u)_{\Omega\cap B(x,s)}\right|^2\,dy\right]^{\frac{1}{2}}
+C_{(n,\mu_0,\theta)}\sigma(s)\fint_{\Omega\cap B(x,s)}|\nabla u(y)|\,dy\\ \notag
&\qquad+C_{(n,\mu_0,\theta)}\sigma(s)\left[\fint_{\Omega\cap B(x,s)}|\boldsymbol{f}(y)|^{q}\,dy\right]^{\frac{1}{q}}
+C_{(n,\mu_0,\theta)}\|\boldsymbol{f}\|_{\mathrm{BMO}(\Omega;{\mathbb{R}^n})},
\end{align}
where $(\nabla u)_{\Omega\cap B(x,s)}:=((u_{x_1})_{\Omega\cap B(x,s)},\ldots,(u_{x_{n}})_{\Omega\cap B(x,r)})$,
the function $\sigma$ is the same as in Assumption \ref{a1.1}, and $C_{(n,\mu_0,\theta)}$ is a positive constant
depending on $n$, $\mu_0$, and $\theta$.
\end{theorem}
\begin{proof}
Without loss of generality, we may assume that $x:={\bf0}$ and, for simplicity, we denote the ball $B({\bf 0},r)$
by $B_r$ throughout this proof. Assume that $q\in(2,q_0)$, where $q_0$ is the exponent appearing in
Proposition \ref{p4.1}. Let $\psi$ be the same as in \eqref{e4.10} and $\sigma$ the same as in Assumption \ref{a1.1}.
By the assumption that $\Omega$ satisfies Assumption \ref{a1.1}, we see that $\psi\in W^1\mathcal{L}^{\sigma(\cdot)}$,
which, combined with Lemma \ref{l2.1}, implies that there exists a positive constant $C$ such that
\begin{equation}\label{e4.26}
\sup_{r\in(0,R_0)}\frac{1}{\sigma(r)}\left[\fint_{\Omega\cap B_r}|{\bf J}(x)-({\bf J})_{\Omega\cap B_r}|^{
2(\frac{q}{2})'}\,dx\right]^{\frac{1}{2(\frac{q}{2})'}}\le C,
\end{equation}
where ${\bf J}$ is the same as in \eqref{e4.11}.

Let $\theta\in(0,1)$ and $s\in (0,R_2]$. From \eqref{e4.22}, we deduce that $\overline{u}$,
defined as in \eqref{e4.20}, is a weak solution of the Neumann problem
\begin{equation}\label{e4.27}
\left\{\begin{array}{ll}
\mathrm{div}_z\left(A_0\nabla_z\overline{u}\right)=\mathrm{div}\left(A_0\nabla_z\overline{u}-A_{{\bf J}_s^{-1}
\overline{{\bf J}}}\nabla_z\overline{u}+\underline{\boldsymbol{f}}\right)\ \
&\text{in}\ H_{{\bf J}_s}\cap B_{\frac{\lambda}{\Lambda}R_2},\\
\dfrac{\partial \overline{u}}{\partial\boldsymbol{\nu}}=\underline{\boldsymbol{f}}
\cdot\boldsymbol{\nu}\  &\text{on}\  \left\{({\bf J}_sz)_n=0\right\}\cap B_{\frac{\lambda}{\Lambda}R_2},
\end{array}\right.
\end{equation}
where $A_0:=(A)_{\Omega\cap B_s}$ and $\underline{\boldsymbol{f}}:=\overline{{\boldsymbol{f}}}\,
\overline{{\bf J}^t}({\bf J}_s^{-1})^t$. Let $D_s:=H_{{\bf J}_s}\cap B_s$, $\theta D_s:=H_{{\bf J}_s}\cap B_{\theta s}$,
and ${\bf J}_s$ be the same as in \eqref{e4.19}. Then, by \eqref{e4.22}, \eqref{e4.23}, and Lemma
\ref{l4.3},  we conclude that there exists a positive constant $C$ independent of $\theta$ such that, for
any $\boldsymbol{f}_0\in{\mathbb{R}^n}$,
\begin{align}\label{e4.28}
&\left[\fint_{\theta D_s}\left|\nabla\overline{u}(z)-(\nabla\overline{u})_{\theta D_{s}}\right|^2\,dz\right]^{\frac{1}{2}}\\ \notag
&\quad\le C\theta\left[\fint_{D_s}\left|\nabla\overline{u}(z)-(\nabla\overline{u})_{D_{s}}\right|^2\,dz\right]^{\frac{1}{2}}
+C_{(n,\mu_0,\theta)}\left[\fint_{D_s}\left|\underline{\boldsymbol{f}}(z)-\boldsymbol{f}_0\right|^2\,dz\right]^{\frac{1}{2}}\\ \notag
&\qquad+C_{(n,\mu_0,\theta)}\left[\fint_{D_s}\left|A_0\nabla\overline{u}(z)-A_{{\bf J}_s^{-1}\overline{{\bf J}}}(z)
\nabla\overline{u}(z)\right|^2\,dz\right]^{\frac{1}{2}},
\end{align}
where $(\nabla\overline{u})_{D_{s}}:=((\overline{{u}}_{z_1})_{D_s},\ldots,(\overline{{u}}_{z_{n}})_{D_s})$ and
$C_{(n,\mu_0,\theta)}$ is a positive constant depending only on $n$, $\mu_0$, and $\theta$.

Fix $r$ such that $\frac{\Lambda}{\lambda}r=\theta s$. From \eqref{e4.13} and \eqref{e4.12}, it follows that
$${\bf J}_s^{-1}\Psi(\Omega\cap B_r)\subset {\bf J}_s^{-1}B_{\Lambda r}^+=H_{{\bf J}_s}\cap B_{\Lambda r}\subset
H_{{\bf J}_s}\cap B_{\frac{\Lambda}{\lambda}r}=H_{{\bf J}_s}\cap B_{\theta s}=\theta D_s,$$
which, together with the fact that $\det{(\nabla\Psi^{-1}({\bf J})_s)}
=\det{(\nabla\Psi^{-1})}\det{({\bf J}_s)}=1$, a change of variables, and \eqref{e4.28}, further implies that
\begin{align}\label{e4.29}
&\left[\fint_{\Omega\cap B_r}\left|\nabla u(x)-(\nabla u)_{\Omega\cap B_r}\right|^2\,dx\right]^{\frac{1}{2}}\\ \notag
&\quad\le 2\left[\fint_{\Omega\cap B_r}
\left|\nabla u(x)-(\nabla\overline{u})_{\theta D_{s}}\right|^2\,dx\right]^{\frac{1}{2}}\\ \notag
&\quad=2\left[\fint_{({\bf J_s})^{-1}\Psi(\Omega\cap B_r)}
\left|\nabla u\left(\Psi^{-1}\left({\bf J}_sz\right)\right)-(\nabla\overline{u})_{\theta D_{s}}\right|^2\,dz\right]^{\frac{1}{2}}\\ \notag
&\quad\le 2\left[\fint_{\theta D_s}\left|\nabla\overline{u}(z)-(\nabla\overline{u})_{\theta D_{s}}\right|^2\,dz\right]^{\frac{1}{2}}
+2\left[\fint_{\theta D_s}\left|\nabla\overline{u}(z)-\nabla u\left(\Psi^{-1}({\bf J}_sz)\right)\right|^2\,dz\right]^{\frac{1}{2}}\\ \notag
&\quad\le 2\left[\fint_{\theta D_s}\left|\nabla\overline{u}(z)-\nabla u\left(\Psi^{-1}({\bf J}_sz)\right)\right|^2\,dz\right]^{\frac{1}{2}}+C\theta\left[\fint_{D_s}\left|\nabla\overline{u}(z)-(\nabla\overline{u})_{D_{s}}\right|^2\,dz\right]^{\frac{1}{2}}\\ \notag
&\qquad +C_{(n,\mu_0,\theta)}\left[\fint_{D_s}\left|\underline{\boldsymbol{f}}(z)-\boldsymbol{f}_0\right|^2\,dz\right]^{\frac{1}{2}}
+C_{(n,\mu_0,\theta)}\left[\fint_{D_s}\left|A_0\nabla\overline{u}(z)-A_{{\bf J}_s^{-1}\overline{{\bf J}}}(z)
\nabla\overline{u}(z)\right|^2\,dz\right]^{\frac{1}{2}}.
\end{align}

Next, we estimate each term on the right-hand side of \eqref{e4.29}. Observe that
\begin{equation}\label{e4.30}
\nabla_z\overline{u}(z)=\left(\nabla_x u\right)\left(\Psi^{-1}\left({\bf J}_s z\right)\right){\bf J}^{-1}\left({\bf J}_s z\right){\bf J}_s,
\end{equation}
\begin{equation*}
{\bf J}^{-1}\left({\bf J}_s z\right){\bf J}\left(\Psi^{-1}({\bf J}_s z)\right)={\bf I}_n,
\end{equation*}
and
\begin{equation}\label{e4.31}
\Psi^{-1}({\bf J}_sD_s)=\Psi^{-1}\left({\bf J}_s\left(H_{{\bf J}_s}\cap B_s\right)\right)=\Psi^{-1}
\left(B_s^+\right)\subset\Omega\cap B_{\frac{s}{\lambda}}
\subset\Omega\cap B_{\frac{\Lambda}{\lambda}s}.
\end{equation}
These, combined with \eqref{e4.2},  \eqref{e4.26}, H\"older's inequality, and the facts that ${\bf J}_s$
and ${\bf J}^{-1}$ are bounded, further imply that there exists a positive constant $C$ such that
\begin{align}\label{e4.32}
&\left[\fint_{\theta D_s}\left|\nabla\overline{u}(z)-\nabla u\left(\Psi^{-1}({\bf J}_sz)\right)\right|^2\,dz\right]^{\frac{1}{2}}\\ \notag
&\quad\le C\left[\fint_{\theta D_s}
\left|{\bf J}\left(\Psi^{-1}({\bf J}_s z)\right)-{\bf J}_s\right|^2\left|\nabla u\left(\Psi^{-1}({\bf J}_s z)\right)\right|^2\,dz\right]^{\frac{1}{2}}\\ \notag
&\quad\le C\left[\fint_{D_s}\left|{\bf J}\left(\Psi^{-1}({\bf J}_s z)\right)-{\bf J}_s\right|^{2(\frac q2)'}dz\right]
^{\frac{1}{2(\frac q2)'}}\left[\fint_{D_s}
\left|\nabla u\left(\Psi^{-1}({\bf J}_s z)\right)\right|^{q}\,dz\right]^{\frac{1}{q}}\\ \notag
&\quad\le C\sigma(s)\left[\fint_{\Omega\cap B_{\frac{\Lambda}{\lambda}s}}|\nabla u(z)|^{q}\,dx\right]^{\frac{1}{q}}\\ \notag
&\quad\le C\sigma(s) \fint_{\Omega\cap B_{2\frac{\Lambda}{\lambda}s}}|\nabla u(x)|\,dx+C\sigma(s)
\left[\fint_{\Omega\cap B_{2\frac{\Lambda}{\lambda}s}}|\boldsymbol{f}(x)|^{q}\,dx\right]^{\frac{1}{q}}.
\end{align}
Similarly, there exists a positive constant $C$ such that
\begin{align}\label{e4.33}
&\left[\fint_{D_s}\left|\nabla\overline{u}(z)-(\nabla\overline{u})_{D_{s}}\right|^2\,dz\right]^{\frac{1}{2}}\\ \notag
&\quad\le 2\left[\fint_{D_s}\left|\nabla\overline{u}(z)-(\nabla u)_{\Omega\cap B_{\frac{\Lambda}{\lambda}s}}\right|^2\,dz\right]^{\frac{1}{2}}\\ \notag
&\quad\le 2\left[\fint_{D_s}\left|\nabla\overline{u}(z)-\nabla u
\left(\Psi^{-1}({\bf J}_sz)\right)\right|^2\,dz\right]^{\frac{1}{2}}\\ \notag
&\qquad+2\left[\fint_{D_s}\left|\nabla u\left(\Psi^{-1}({\bf J}_sz)\right)
-(\nabla u)_{\Omega\cap B_{\frac{\Lambda}{\lambda}s}}\right|^2\,dz\right]^{\frac{1}{2}}\\ \notag
&\quad\le C\left[\fint_{\Omega\cap B_{\frac{\Lambda}{\lambda}s}}\left|\nabla u(x)-(\nabla u)_{\Omega\cap B_{\frac{\Lambda}{\lambda}s}}\right|^{2}\,dx\right]^{\frac{1}{2}}+C\sigma(s)\left[\fint_{\Omega\cap B_{2\frac{\Lambda}{\lambda}s}}
|\boldsymbol{f}(x)|^{q}\,dx\right]^{\frac{1}{q}}\\ \notag
&\qquad+C\sigma(s)\fint_{\Omega\cap B_{2\frac{\Lambda}{\lambda}s}}|\nabla u(x)|\,dx.
\end{align}
By the definition of $\underline{\boldsymbol{f}}$ in \eqref{e4.27}, \eqref{e4.26}, Assumption \ref{a1.1}(a),
the boundedness of ${\bf J},{\bf J}^{-1}$, and ${\bf J}_s$, and Lemma \ref{l2.4}, we obtain that
\begin{align}\label{e4.34}
\inf_{\boldsymbol{f}_0\in{\mathbb{R}^n}}\left[\fint_{D_s}\left|\underline{\boldsymbol{f}}(z)-
\boldsymbol{f}_0\right|^2\,dz\right]^{\frac{1}{2}}\lesssim\left\|\overline{{\boldsymbol{f}}}\,
\overline{{\bf J}^t}\right\|_{\mathrm{BMO}(D_{R_2};{\mathbb{R}^n})}\lesssim\left\|\overline{{\boldsymbol{f}}}
\right\|_{\mathrm{BMO}(D_{R_2};{\mathbb{R}^n})}\lesssim\left\|{\boldsymbol{f}}\right\|_{\mathrm{BMO}(\Omega;{\mathbb{R}^n})}.
\end{align}
Furthermore, observe that, for any $n\times n$ matrix ${\bf D},{\bf E},{\bf F},{\bf H}$,
\begin{align*}
{\bf F}{\bf E}{\bf F}^t-{\bf H}{\bf D}{\bf H}^t={\bf F}({\bf E}-{\bf D}){\bf F}^t
+({\bf F}-{\bf H}){\bf D}{\bf F}^t+{\bf H}{\bf D}({\bf F}^t-{\bf H}^t).
\end{align*}
From this, \eqref{e4.2}, \eqref{e4.30}, \eqref{e4.31}, Lemma \ref{l2.1}, H\"older's inequality,
and Assumption \ref{a1.1}$\mathrm{(a)}$, we deduce that
there exists a positive constant $C$ such that
\begin{align}\label{e4.35}
&\left[\fint_{D_s}\left|A_0\nabla\overline{u}(z)-A_{{\bf J}_s^{-1}\overline{{\bf J}}}\nabla\overline{u}(z)
\right|^2\,dz\right]^{\frac{1}{2}}\\ \notag
&\quad=\left[\fint_{D_s}\left|\left(A_0-{\bf J}_s^{-1}{\bf J}\left(\Psi^{-1}({\bf J}_sz)\right)A\left(\Psi^{-1}
({\bf J}_sz)\right){\bf J}^t\left(\Psi^{-1}({\bf J}_sz)\right)
\left({\bf J}_s^{-1}\right)^t\right)\nabla\overline{u}(z)\right|^2\,dz\right]^{\frac{1}{2}}\\ \notag
&\quad\le C\left[\fint_{D_s}|\nabla\overline{u}(z)|^{q}\,dz\right]^{\frac{1}{q}}\\ \notag
&\qquad\times\left[\fint_{D_s}\left|{\bf J}_sA_0{\bf J}^t_s-
{\bf J}\left(\Psi^{-1}({\bf J}_sz)\right)A\left(\Psi^{-1}({\bf J}_sz)\right){\bf J}^t
\left(\Psi^{-1}({\bf J}_sz)\right)\right|^{2(\frac{q}{2})'}\,dz\right]^{\frac{1}{2(\frac{q}{2})'}}\\ \notag
&\quad\le C\left[\fint_{D_s}|\nabla\overline{u}(z)|^{q}\,dz\right]^{\frac{1}{q}}\left[\fint_{D_s}\left|{\bf J}
\left(\Psi^{-1}({\bf J}_s z)\right)-{\bf J}_s\right|^{2(\frac{q}{2})'}\,dz\right]^{\frac{1}{2(\frac{q}{2})'}}\\ \notag
&\quad\quad+C\left[\fint_{D_s}|\nabla\overline{u}(z)|^{q}\,dz\right]^{\frac{1}{q}}\left[\fint_{D_s}\left|A_0-A\left(\Psi^{-1}
({\bf J}_sz)\right)\right|^{2(\frac{q}{2})'}\,dz\right]^{\frac{1}{2(\frac{q}{2})'}}\\ \notag
&\quad\le C\sigma(s)\left\{\fint_{\Omega\cap B_{2\frac{\Lambda}{\lambda}s}}|\nabla u(x)|\,dx
+\left[\fint_{\Omega\cap B_{2\frac{\Lambda}{\lambda}s}}|\boldsymbol{f}(x)|^{q}\,dx\right]^{\frac{1}{q}}\right\}.
\end{align}
Combining \eqref{e4.29}, \eqref{e4.32}, \eqref{e4.33}, \eqref{e4.34}, and
\eqref{e4.35}, we conclude that there exists a positive constant $C$ independent of $\theta$ such that
\begin{align}\label{e4.36}
&\left[\fint_{\Omega\cap B_{\frac{\lambda}{\Lambda}\theta s}}\left|\nabla u(x)-
(\nabla u)_{\Omega\cap B_{\frac{\Lambda}{\lambda}\theta s}}\right|^2\,dx\right]^{\frac{1}{2}}\\ \notag
&\quad\le C_{(n,\mu_0,\theta)}\sigma(s)\fint_{\Omega\cap B_{2\frac{\Lambda}{\lambda}s}}|\nabla u(x)|\,dx
+C_{(n,\mu_0,\theta)}\sigma(s)\left[\fint_{\Omega\cap B_{2\frac{\Lambda}{\lambda}s}}
\left|\boldsymbol{f}(x)\right|^{q}\,dx\right]^{\frac{1}{q}}\\ \notag
&\qquad+C_{(n,\mu_0,\theta)}\|\boldsymbol{f}\|_{\mathrm{BMO}(\Omega;{\mathbb{R}^n})}
+C\theta\left[\fint_{\Omega\cap B_{\frac{\Lambda}{\lambda}s}}\left|\nabla u(x)-
(\nabla u)_{\Omega\cap B_{\frac{\Lambda}{\lambda}s}}\right|^2\,dx\right]^{\frac{1}{2}}.
\end{align}
Therefore, the estimate \eqref{e4.25} follows from \eqref{e4.36} by redefining the parameter $\theta$.
This finishes the proof of Theorem \ref{t4.1}.
\end{proof}

Using the interior estimate \eqref{e3.4}, Proposition \ref{p4.2}, Lemma \ref{l4.4}, and an argument similar to that
used in the proof of Theorem \ref{t4.1}, we obtain the following decay estimate at the
boundary for the gradient of the solution $u$ to the Dirichlet problem \eqref{e1.8}; we omit its proof.

\begin{theorem}\label{t4.2}
Let $n\ge2$ and $\Omega$ be a bounded Lipschitz domain of ${\mathbb{R}^n}$. Assume that ${\bf 0}\in\partial\Omega$. Let $R_2:=\min\{R_0,R_1\}$,
where $R_0$ is the same as in Assumption \ref{a1.1} and $R_1$ the same as in \eqref{e4.10}.
Assume that the matrix $A$ and the domain $\Omega$ satisfy Assumption \ref{a1.1},
and local coordinates in $\Omega\cap B({\bf0},R_1)$ are the same as in Subsection \ref{s4.2}. Assume further that $\theta\in(0,1)$,
$\boldsymbol{f}\in \mathrm{BMO}_r(\Omega;{\mathbb{R}^n})$, and $u$ is the weak solution to the Dirichlet problem
\eqref{e1.8}. Then there exist constants $C\in(0,\infty)$ and $q\in(2,\infty)$, independent of $\theta$,
such that, for any given $s\in(0,R_2]$ and any $\boldsymbol{f}_0\in{\mathbb{R}^n}$,
\begin{align*}
\left[\fint_{\Omega\cap B({\bf0},\theta s)}\left|\nabla u(y)-\boldsymbol{U}_{\theta s}\right|^2\,dy\right]^{\frac{1}{2}}&\le
C\theta\left[\fint_{\Omega\cap B({\bf0},s)}\left|\nabla u(y)-\boldsymbol{U}_{s}\right|^2\,dy\right]^{\frac{1}{2}}\\
&\quad+C_{(n,\mu_0,\theta)}\sigma(s)\fint_{\Omega\cap B({\bf0},s)}
|\nabla u(y)|\,dy\\ \notag
&\quad+C_{(n,\mu_0,\theta,R_2)}\left[\fint_{\Omega\cap B({\bf0},s)}|\boldsymbol{f}(y)-\boldsymbol{f}_0|^{q}\,dy\right]^{\frac{1}{q}},
\end{align*}
where $\boldsymbol{U}_s:=(0,\ldots,0,(u_{x_n})_{\Omega\cap B({\bf0},s)})$, the function $\sigma$ is the same as in Assumption \ref{a1.1},
$C_{(n,\mu_0,\theta)}$ is a positive constant depending only on $n$, $\mu_0$, and $\theta$, and $C_{(n,\mu_0,\theta,R_2)}$ is a positive constant depending only on $n$, $\mu_0$, $\theta$, and $R_2$.
\end{theorem}

\section{Proofs of Theorem \ref{th1.1} and Corollary \ref{c1.1}}\label{s5}

In this section, we prove Theorem \ref{th1.1} and Corollary \ref{c1.1} by using Theorems
\ref{t4.1} and \ref{t4.2}, Proposition \ref{p4.3}, and Lemma \ref{l2.2}.

\begin{theorem}\label{t5.1}
Let $n\ge2$ and $\Omega\subset{\mathbb{R}^n}$ be a bounded Lipschitz domain.
Assume that the matrix $A$ and the domain $\Omega$ satisfy Assumption \ref{a1.1},
$\boldsymbol{f}\in \mathrm{BMO}_r(\Omega;{\mathbb{R}^n})$, and $u$ is the weak solution to the Neumann problem
\eqref{e1.6}. Then there exists a constant $R_3\in(0,1)$, depending on $n$, $\mu_0$, the function $\sigma$ as in Assumption \ref{a1.1}, and $\mathrm{diam\,}(\Omega)$, such that,
for any $x\in\partial\Omega$ and $R\in(0,R_3]$,
\begin{equation}\label{e5.1}
\sup_{s\in(0,R)}\left[\fint_{\Omega\cap B(x,s)}\left|\nabla u(y)-(\nabla u)_{\Omega\cap B(x,s)}\right|^2\,dy\right]^{\frac{1}{2}}
\le C\|\boldsymbol{f}\|_{\mathrm{BMO}_{r,+}(\Omega;{\mathbb{R}^n})},
\end{equation}
where $C$ is a positive constant independent of $\boldsymbol{f},$ $u,$ and $x$.
\end{theorem}
To show Theorem \ref{t5.1}, we need the following well-known lemma (see, for instance, \cite[Lemma 5.2]{bc22}).

\begin{lemma}\label{l5.1}
Let $n\ge2$ and $\Omega\subset{\mathbb{R}^n}$ be a bounded Lipschitz domain. Assume that $q\in[1,\infty)$,
$f\in L^q(\Omega)$, and $R\in(0,\mathrm{diam\,}(\Omega)]$ be a given constant. Then there exists
a positive constant $C$, depending only on $n,$ $q$, and the Lipschitz constant of $\Omega$, such that,
for any $x\in\overline{\Omega}$ and $r\in(0,R]$,
\begin{equation*}
\fint_{\Omega\cap B(x,r)}|f(y)|\,dy\le C\ln\left({\frac{R}{r}}\right)\sup_{\rho\in(r,R)}\left[\fint_{\Omega\cap
B(x,\rho)}\left|f(y)-(f)_{\Omega\cap B(x,\rho)}\right|^q\,dy\right]^{\frac{1}{q}}+C(|f|)_{\Omega\cap B(x,R)}.
\end{equation*}
\end{lemma}

Now, we prove Theorem \ref{t5.1} by using Proposition \ref{p4.1}, Theorem \ref{t4.1},
and Lemma \ref{l5.1}.

\begin{proof}[Proof of Theorem \ref{t5.1}]
Let $x\in\partial\Omega$ and $\theta\in(0,1)$ be determined later. Assume that $R_2\in(0,\infty)$
is the same as in Theorem \ref{t4.1}. By Theorem \ref{t4.1}, there exists a positive constant
$C_3$, independent of $x,$ $u,$ $\boldsymbol{f}$, and $\theta$, such that, for any $s\in(0,R_2]$,
\begin{align}\label{e5.2}
&\left[\fint_{\Omega\cap B(x,\theta s)}\left|\nabla u(y)-(\nabla u)_{\Omega\cap B(x,\theta s)}\right|^2\,dy\right]^{\frac{1}{2}}\\ \notag
&\quad\le C_{(n,\mu_0,\theta)}\sigma(s)\fint_{\Omega\cap B(x,s)}|\nabla u(y)|\,dy
+ C_{(n,\mu_0,\theta)}\sigma(s)\left[\fint_{\Omega\cap B(x,s)}|\boldsymbol{f}(y)|^{q}\,dy\right]^{\frac{1}{q}}\\ \notag
&\qquad+ C_{(n,\mu_0,\theta)}\|\boldsymbol{f}\|_{\mathrm{BMO}(\Omega;{\mathbb{R}^n})}
+C_3\theta\left[\fint_{\Omega\cap B(x,s)}\left|\nabla u(y)-(\nabla u)_{\Omega\cap B(x,s)}\right|^2\,dy\right]^{\frac{1}{2}},
\end{align}
where $C_{(n,\mu_0,\theta)}$ is a positive constant depending on $n$, $\mu_0$, and $\theta$ and $q\in(2,\infty)$ is the same as in
Theorem \ref{t4.1}.

Take $\theta\in(0,1)$ such that
\begin{align}\label{e5.3}
C_3\theta\le1/2.
\end{align}
From Lemma \ref{l2.1}, it follows that, for any $s\in(0,R_2]$,
\begin{align}\label{e5.4}
\left[\fint_{\Omega\cap B(x,s)}|\boldsymbol{f}(y)|^{q}\,dy\right]^{\frac{1}{q}}&\le
\left[\fint_{\Omega\cap B(x,s)}\left|\boldsymbol{f}(y)-(\boldsymbol{f})_{\Omega\cap B(x,s)}\right|^{q}\,dy\right]^{\frac{1}{q}}
+\fint_{\Omega\cap B(x,s)}|\boldsymbol{f}(y)|\,dy\\ \notag
&\lesssim\|\boldsymbol{f}\|_{\mathrm{BMO}(\Omega;{\mathbb{R}^n})}+\fint_{\Omega\cap B(x,s)}|\boldsymbol{f}(y)|\,dy.
\end{align}
Moreover, by Lemma \ref{l5.1}, there exists a positive constant $C_4$ such that,
for any given $g\in L^2(\Omega)$, any $R\in(0,\min\{R_2,\mathrm{diam\,}(\Omega),1\}]$, and $s\in(0,R)$,
\begin{align*}
\fint_{\Omega\cap B(x,s)}|g(y)|\,dy&\le C_4\ln\left({\frac{1}{s}}\right)
\sup_{\rho\in(s,R)}\left[\fint_{\Omega\cap B(x,\rho)}\left|g(y)-(g)_{\Omega\cap B(x,\rho)}\right|^2\,
dy\right]^{\frac{1}{2}}+C_4(|g|)_{\Omega\cap B(x,R)},
\end{align*}
which, combined with \eqref{e5.2} and \eqref{e5.4}, further implies that, for any $R\in(0,\min\{R_2,
\mathrm{diam\,}(\Omega),1\}]$ and $s\in(0,R)$,
\begin{align}\label{e5.5}
&\left[\fint_{\Omega\cap B(x,\theta s)}\left|\nabla u(y)-(\nabla u)_{\Omega\cap B(x,\theta s)}\right|^2\,dy\right]^{\frac{1}{2}}\\ \notag
&\quad\le C_4C_{(n,\mu_0,\theta)}\sigma(s)\ln\left({\frac{1}{s}}\right)\sup_{\rho\in(s,R)}\left[\fint_{\Omega\cap B(x,\rho)}
\left|\nabla u(y)-(\nabla u)_{\Omega\cap B(x,\rho)}\right|^2\,dy\right]^{\frac{1}{2}}\\ \notag
&\qquad+C_4C_{(n,\mu_0,\theta)}\sigma(s)\fint_{\Omega\cap B(x,R)}|\nabla u(y)|\,dy
+ C_{(n,\mu_0,\theta)}\|\boldsymbol{f}\|_{\mathrm{BMO}(\Omega;{\mathbb{R}^n})}\\ \notag
&\qquad+C_4C_{(n,\mu_0,\theta)}\sigma(s)\ln\left({\frac{1}{s}}\right)\sup_{\rho\in(s,R)}\left[\fint_{\Omega\cap B(x,\rho)}
\left|\boldsymbol{f}(y)-(\boldsymbol{f})_{\Omega\cap B(x,\rho)}\right|^2\,dy\right]^{\frac{1}{2}}\\ \notag
&\qquad+C_{(n,\mu_0,\theta)}\sigma(s)\fint_{\Omega\cap B(x,R)}|\boldsymbol{f}(y)|\,dy+C_3\theta\left[\fint_{\Omega\cap B(x,s)}\left|\nabla u(y)-(\nabla u)_{\Omega\cap B(x,s)}\right|^2\,dy\right]^{\frac{1}{2}}.
\end{align}
From (i) of Assumption \ref{a1.1}(a), we deduce that there exists a positive constant $\widetilde{R}\in(0,1)$, depending on $n$, $\mu_0$, $\theta$, and the function
$\sigma$, such that
\begin{align}\label{e5.6}
\sup_{s\in(0,\widetilde{R})}C_4C_{(n,\mu_0,\theta)}\sigma(s)\ln\left({\frac{1}{s}}\right)\le\frac14.
\end{align}
Let $R_3:=\min\{R_2,\widetilde{R},\mathrm{diam\,}(\Omega)\}$. By \eqref{e5.3}, \eqref{e5.5}, \eqref{e5.6}, and Lemma \ref{l2.2},
for any given $R\in(0,R_3]$ and $s\in(0,R)$,
\begin{align*}
&\left[\fint_{\Omega\cap B(x,\theta s)}\left|\nabla u(y)-(\nabla u)_{\Omega\cap B(x,\theta s)}\right|^2\,dy\right]^{\frac{1}{2}}\\ \notag
&\quad\le\frac{3}{4}\sup_{\rho\in(s,R)}\left[\fint_{\Omega\cap B(x,\rho)}
\left|\nabla u(y)-(\nabla u)_{\Omega\cap B(x,\rho)}\right|^2\,dy\right]^{\frac{1}{2}}\\ \notag
&\qquad+C\fint_{\Omega\cap B(x,R)}|\nabla u(y)|\,dy+C\fint_{\Omega\cap B(x,R)}\left|\boldsymbol{f}(y)\right|\,dy
+C\|\boldsymbol{f}\|_{\mathrm{BMO}(\Omega;{\mathbb{R}^n})}\\ \notag
&\quad\le\frac{3}{4}\sup_{\rho\in(s,R)}\left[\fint_{\Omega\cap B(x,\rho)}\left|\nabla u(y)-(\nabla u)_{\Omega\cap B(x,\rho)}\right|^2\,dy
\right]^{\frac{1}{2}}\\
&\qquad+C\fint_{\Omega\cap B(x,R)}|\nabla u(y)|\,dy+C\|\boldsymbol{f}\|_{\mathrm{BMO}_{r,+}(\Omega;{\mathbb{R}^n})},
\end{align*}
which, together with H\"older's inequality and Remark \ref{r1.1}(i), further implies that, for any given $\varepsilon\in(0,\theta R/2)$,
\begin{align}\label{e5.7}
&\sup_{s\in(\varepsilon,\frac{\theta R}{2})}\left[\fint_{\Omega\cap B(x,s)}
\left|\nabla u(y)-(\nabla u)_{\Omega\cap B(x,s)}\right|^2\,dy\right]^{\frac{1}{2}}\\ \notag
&\quad\le\frac{3}{4}\sup_{\rho\in(\frac{\varepsilon}{\theta},R)}\left[\fint_{\Omega\cap B(x,\rho)}\left|\nabla u(y)
-(\nabla u)_{\Omega\cap B(x,\rho)}\right|^2\,dy\right]^{\frac{1}{2}}+C\|\boldsymbol{f}\|_{\mathrm{BMO}_{r,+}(\Omega;{\mathbb{R}^n})}.
\end{align}
Moreover, from Remark \ref{r1.1}(i), it follows that
\begin{align}\label{e5.8}
&\sup_{s\in(\frac{\theta R}{2},R)}\left[\fint_{\Omega\cap B(x,s)}\left|\nabla u(y)-
(\nabla u)_{\Omega\cap B(x,s)}\right|^2\,dy\right]^{\frac{1}{2}}\\ \nonumber
&\quad\le C\sup_{s\in(\frac{\theta R}{2},R)}
\left[\fint_{\Omega\cap B(x,s)}|\nabla u(y)|^2\,dy\right]^{\frac{1}{2}}
\le C\left[\fint_{\Omega\cap B(x,R)}|\nabla u(y)|^2\,dy\right]^{\frac{1}{2}}\le C\|\boldsymbol{f}\|_{L^2(\Omega;{\mathbb{R}^n})}.
\end{align}
Then, by \eqref{e5.7} and \eqref{e5.8}, for any given $\varepsilon\in(0,\theta R/2)$,
\begin{align*}
&\sup_{s\in(\varepsilon,R)}\left[\fint_{\Omega\cap B(x,s)}\left|\nabla u(y)-(\nabla u)_{\Omega\cap B(x,s)}\right|^2\,dy\right]^{\frac{1}{2}}\\ \notag
&\quad\le\frac{3}{4}\sup_{\rho\in(\frac{\varepsilon}{\theta},R)}\left[\fint_{\Omega\cap B(x,\rho)}\left|\nabla u(y)
-(\nabla u)_{\Omega\cap B(x,\rho)}\right|^2\,dy\right]^{\frac{1}{2}}+C\|\boldsymbol{f}\|_{\mathrm{BMO}_{r,+}(\Omega;{\mathbb{R}^n})},
\end{align*}
which further implies that, for any given $\varepsilon\in(0,\theta R/2)$,
\begin{align}\label{e5.9}
\sup_{s\in(\varepsilon,R)}\left[\fint_{\Omega\cap B(x,s)}\left|\nabla u(y)-(\nabla u)_{\Omega\cap B(x,s)}\right|^2\,dy\right]^{\frac{1}{2}}
\lesssim\|\boldsymbol{f}\|_{\mathrm{BMO}_{r,+}(\Omega;{\mathbb{R}^n})}.
\end{align}
Then letting $\varepsilon\to0$ in \eqref{e5.9}, we conclude that \eqref{e5.1} holds. This finishes the proof of Theorem \ref{t5.1}.
\end{proof}

Applying Theorem \ref{t4.2} and an argument similar to that used in the proof of
Theorem \ref{t5.1}, we obtain the following estimate for the Dirichlet problem \eqref{e1.8}; we omit its proof.

\begin{theorem}\label{t5.2}
Let $n\ge2$ and $\Omega\subset{\mathbb{R}^n}$ be a bounded Lipschitz domain. Assume that the matrix $A$ and the domain $\Omega$ satisfy Assumption \ref{a1.1},
$\boldsymbol{f}\in \mathrm{BMO}_r(\Omega;{\mathbb{R}^n})$, and $u$ is the weak solution to the Dirichlet problem \eqref{e1.8}.
Then there exists a constant $R_4\in(0,1)$, depending on $n$, $\mu_0$, the function $\sigma$ in Assumption \ref{a1.1}, and $\mathrm{diam\,}(\Omega)$,
such that, for any $x\in\partial\Omega$ and $R\in(0,R_4]$,
\begin{equation*}
\sup_{s\in(0,R)}\left[\fint_{\Omega\cap B(x,s)}\left|\nabla u(y)-(\nabla u)_{\Omega\cap B(x,s)}\right|^2\,dy\right]^{\frac{1}{2}}
\le C\|\boldsymbol{f}\|_{\mathrm{BMO}_{r}(\Omega;{\mathbb{R}^n})},
\end{equation*}
where $C$ is a positive constant independent of $\boldsymbol{f},$ $u$, and $x$.
\end{theorem}

Applying Proposition \ref{p4.3}, Lemma \ref{l5.1}, and an argument similar to that used in the proof of
Theorem \ref{t5.1}, we obtain the following proposition; we omit the proof.

\begin{proposition}\label{p5.1}
Let $n\ge2$ and $\Omega\subset{\mathbb{R}^n}$ be a bounded Lipschitz domain.
Assume that the matrix $A$ and the domain $\Omega$ satisfy Assumption \ref{a1.1},
$\boldsymbol{f}\in \mathrm{BMO}_r(\Omega;{\mathbb{R}^n})$, and $u$ is the weak solution to the Neumann problem \eqref{e1.6}
or the Dirichlet problem \eqref{e1.8}. Then there exist positive constants $\delta_0,$ $R_5\in(0,1)$,
depending on $n$, $\mu_0$, the function $\sigma$ in Assumption \ref{a1.1}, and $\mathrm{diam\,}(\Omega)$, such that, for
any given $R\in(0,R_5]$ and any $x\in\Omega$ and $r\in(0,R)$ satisfying $B(x,2r)\subset\Omega$,
\begin{align}\label{e5.10}
&\left[\fint_{B(x,\delta_0r)}\left|\nabla u(y)-(\nabla u)_{B(x,\delta_0r)}\right|^2\,dy\right]^{\frac{1}{2}}\\ \nonumber
&\quad\le\frac{1}{2}\sup_{s\in(0,R)}\left[\fint_{\Omega\cap B(x,s)}\left|\nabla u(y)-(\nabla u)_{\Omega\cap
B(x,s)}\right|^2\,dy\right]^{\frac{1}{2}}+C\|\boldsymbol{f}\|_{\mathrm{BMO}_{r,+}(\Omega;{\mathbb{R}^n})}<\infty,
\end{align}
where $C$ is a positive constant independent of $\boldsymbol{f}$, $u$, and $x$.
\end{proposition}

Now, we show Theorem \ref{th1.1} by using Lemma \ref{l2.3},
Theorems \ref{t5.1} and \ref{t5.2}, and Proposition \ref{p5.1}.

\begin{proof}[Proof of Theorem \ref{th1.1}]
We only give the proof of \eqref{e1.10} in the case of the Neumann problem because the proof
in the case of the Dirichlet problem is similar.  Let $\boldsymbol{f}\in\mathrm{BMO}_{r}(\Omega;{\mathbb{R}^n})$,
$u$ be the weak solution to the Neumann problem \eqref{e1.6}, and $R:=\min\{R_3, R_5\}$, where
$R_3$ is the same as in Theorem \ref{t5.1} and $R_5$ the same as in Proposition \ref{p5.1}.
Without loss of generality, we may assume that $\int_\Omega u\,dx=0$.
Take $c_0\in(0,\delta_0/16)$, where $\delta_0$ is the same as in Proposition \ref{p5.1}.

From H\"older's inequality, Remark \ref{r1.1}, and the definition of $\|\boldsymbol{f}\|_{
\mathrm{BMO}_{r,+}(\Omega;{\mathbb{R}^n})}$, we deduce that
\begin{align}\label{e5.12}
&\sup_{x\in\Omega}\sup_{r\in[c_0R,\mathrm{diam\,}(\Omega))}\left[\fint_{\Omega\cap B(x,r)}\left|\nabla u(y)-
(\nabla u)_{\Omega\cap B(x,r)}\right|^2\,dy\right]^{1/2}\\ \notag
&\quad\lesssim\left[\int_{\Omega}|\nabla u(y)|^2\,dy\right]^{1/2}
\lesssim\|\boldsymbol{f}\|_{\mathrm{BMO}_{r,+}(\Omega;{\mathbb{R}^n})}.
\end{align}
Now, assume that $x\in\Omega$ and $r\in(0,c_0R)$. If $B(x,\frac{2}{\delta_0}r)\cap\partial\Omega\neq\emptyset$,
then there exists $x_0\in\partial\Omega$ such that $B(x,r)\subset B(x_0,\frac{8}{\delta_0}r)$.
By the assumptions $c_0<\frac{\delta_0}{16}$ and $r\in(0,c_0R)$, we have $\frac{8}{\delta_0}r<\frac{R}{2}$.
Then, from \eqref{e5.1}, it follows that, when $B(x,\frac{2}{\delta_0}r)
\cap\partial\Omega\neq\emptyset$,
\begin{align}\label{e5.13}
&\left[\fint_{\Omega\cap B(x,r)}\left|\nabla u(y)-(\nabla u)_{\Omega\cap B(x,r)}\right|^2\,dy\right]^{1/2}\\ \notag
&\quad\lesssim\left[\fint_{\Omega\cap B(x_0,\frac{8}{\delta_0}r)}\left|\nabla u(y)-(\nabla u)_{\Omega\cap
B(x_0,\frac{8}{\delta_0}r)}\right|^2\,dy\right]^{1/2}\lesssim\|\boldsymbol{f}\|_{\mathrm{BMO}_{r,+}(\Omega;{\mathbb{R}^n})}.
\end{align}
On the other hand, if $B(x,\frac{2}{\delta_0}r)\subset\Omega$, then, by \eqref{e5.10},
\begin{align*}
&\left[\fint_{B(x,r)}|\nabla u(y)-(\nabla u)_{B(x,r)}|^2\,dy\right]^{1/2}\\ \notag
&\quad\le\frac{1}{2}\sup_{s\in(0,R)}\left[\fint_{\Omega\cap B(x,s)}\left|\nabla u(y)-(\nabla u)_{\Omega\cap
B(x,s)}\right|^2\,dy\right]^{\frac{1}{2}}+C\|\boldsymbol{f}\|_{\mathrm{BMO}_{r,+}(\Omega;{\mathbb{R}^n})}<\infty,
\end{align*}
which, combined with \eqref{e5.12} and \eqref{e5.13}, further implies that
there exists a positive constant $C$ independent of $u$ and $\boldsymbol{f}$ such that
\begin{align*}
&\sup_{x\in\Omega}\sup_{r\in(0,\mathrm{diam\,}(\Omega))}
\left[\fint_{\Omega\cap B(x,r)}\left|\nabla u(y)-(\nabla u)_{\Omega\cap B(x,r)}\right|^2\,dy\right]^{1/2}\\
&\quad\le\frac{1}{2}\sup_{x\in\Omega}\sup_{r\in(0,\mathrm{diam\,}(\Omega))}
\left[\fint_{\Omega\cap B(x,r)}\left|\nabla u(y)-(\nabla u)_{\Omega\cap B(x,r)}\right|^2\,dy\right]^{1/2}
+C\|\boldsymbol{f}\|_{\mathrm{BMO}_{r,+}(\Omega;{\mathbb{R}^n})}<\infty.
\end{align*}
From this and Lemmas \ref{l2.1} and \ref{l2.2}, we deduce that $\nabla u\in\mathrm{BMO}_r(\Omega;{\mathbb{R}^n})$
and \eqref{e1.10} holds. This finishes the proof of Theorem \ref{th1.1}.
\end{proof}

\begin{proof}[Proof of Corollary \ref{c1.1}]
Here we only give the proof in the case of the Neumann problem
because the proof in the case of the Dirichlet problem is similar.

We first assume that $\boldsymbol{f}\in H^1_z(\Omega;{\mathbb{R}^n})\cap L^2(\Omega;{\mathbb{R}^n})$.
Let $u\in W^{1,2}(\Omega)$ be the weak solution to the Neumann problem \eqref{e1.6}.
Without loss of generality, we may assume that $\int_\Omega u\,dx=0$. Let $v$ be the weak solution
to the Neumann problem \eqref{e1.6} with the coefficient matrix $A^t$ and the right-hand side
$\boldsymbol{g}\in L^\infty(\Omega;{\mathbb{R}^n})$. Here $A^t$ denotes the transpose of the matrix $A$. Then
\begin{align}\label{e5.19}
\int_{\Omega}\boldsymbol{f}(x)\cdot\nabla v(x)\,dx&=\int_{\Omega}A(x)\nabla u(x)\cdot \nabla v(x)\,dx\\ \nonumber
&=\int_{\Omega}A^t(x)\nabla v(x)\cdot\nabla u(x)\,dx
=\int_{\Omega}\boldsymbol{g}(x)\cdot\nabla u(x)\,dx.
\end{align}
Notice that the matrix $A$ satisfies Assumption \ref{a1.1} if and only if $A^t$ satisfies the same Assumption \ref{a1.1}.
Therefore, by Theorem \ref{th1.1} and the obvious fact that $L^\infty(\Omega)\hookrightarrow\mathrm{BMO}_r(\Omega)$, we have
\begin{align*}
\|\nabla v\|_{\mathrm{BMO}_{r,+}(\Omega;{\mathbb{R}^n})}\lesssim\|\boldsymbol{g}\|_{L^\infty(\Omega;{\mathbb{R}^n})},
\end{align*}
which, together with \eqref{e5.19} and Lemma \ref{l2.3}, further implies that
\begin{align}\label{e5.20}
\|\nabla u\|_{L^1(\Omega;{\mathbb{R}^n})}&=\sup_{\|\boldsymbol{g}\|_{L^\infty(\Omega;{\mathbb{R}^n})}\le1}
\left|\int_{\Omega}\boldsymbol{g}(x)\cdot\nabla u(x)\,dx\right|
=\sup_{\|\boldsymbol{g}\|_{L^\infty(\Omega;{\mathbb{R}^n})}\le1}\left|\int_\Omega\boldsymbol{f}(x)\cdot\nabla v(x)\,dx\right|\\ \nonumber
&\lesssim\sup_{\|\boldsymbol{g}\|_{L^\infty(\Omega;{\mathbb{R}^n})}\le1}
\|\boldsymbol{f}\|_{H^1_z(\Omega;{\mathbb{R}^n})}\|\nabla v\|_{\mathrm{BMO}_{r}(\Omega;{\mathbb{R}^n})}
\lesssim\|\boldsymbol{f}\|_{H^1_z(\Omega;{\mathbb{R}^n})}.
\end{align}
This estimate, combined with the assumption $\int_{\Omega} u\,dx=0$ and the Sobolev--Poincar\'e
inequality, yields that $u\in W^{1,1}(\Omega)$.

Now, assume that $\boldsymbol{f}\in H^1_z(\Omega;{\mathbb{R}^n})$. Since
$H^1_z(\Omega)\cap L^2(\Omega)$ is dense in $H^1_z(\Omega)$ (see, for instance, \cite[p.\,109, Lemma]{St93}), it follows that
there exists a sequence $\{\boldsymbol{f}_m\}_{m=1}^\infty\subset H^1_z(\Omega;{\mathbb{R}^n})
\cap L^2(\Omega;{\mathbb{R}^n})$ such that
\begin{equation}\label{eq5.21}
\lim_{m\to\infty}\|\boldsymbol{f}_m-\boldsymbol{f}
\|_{H^1_z(\Omega;{\mathbb{R}^n})}=0.
\end{equation}
For any $m\in\mathbb{N}$, let $u_m\in W^{1,2}(\Omega)$ be the weak
solution of the Neumann problem \eqref{e1.6} with the right-hand side $\boldsymbol{f}_m$.
Assume also that, for any $m\in\mathbb{N}$, $\int_{\Omega} u_m\,dx=0$. By this, \eqref{e5.20},
and \eqref{eq5.21}, we see that $\{u_m\}_{m=1}^\infty$ is a Cauchy sequence in $W^{1,1}(\Omega)$.
Therefore, there exist a function $u\in W^{1,1}(\Omega)$ and a subsequence of $\{u_m\}_{m=1}^\infty$,
still denoted by $\{u_m\}_{m=1}^\infty$, such that $u_m\to u$ in $W^{1,1}(\Omega)$ as $m\to\infty$.
From this, \eqref{eq5.21}, and the fact that $H^1_z(\Omega)\hookrightarrow L^1(\Omega)$,
we further deduce that, for any $\varphi\in C^{\infty}({\mathbb{R}^n})$,
\begin{align*}
\int_{\Omega}A(x)\nabla u(x)\cdot \nabla\varphi(x)\,dx
=\int_{\Omega}\boldsymbol{f}(x)\cdot\nabla\varphi(x)\,dx,
\end{align*}
which implies that $u\in W^{1,1}(\Omega)$ is a weak solution of the Neumann problem
\eqref{e1.6} with $\boldsymbol{f}\in H^1_z(\Omega;{\mathbb{R}^n})$. By this and \cite[Theorem 1.2]{acmm10},
we conclude that the Neumann problem \eqref{e1.6} with $\boldsymbol{f}\in H^1_z(\Omega;{\mathbb{R}^n})$
is uniquely solvable. Furthermore, from \eqref{e5.20}, \eqref{eq5.21}, and $\lim_{m\to\infty}\|\nabla u_m-\nabla u
\|_{L^1(\Omega;{\mathbb{R}^n})}=0$, it follows that $\|\nabla u\|_{L^1(\Omega;{\mathbb{R}^n})}\lesssim
\|\boldsymbol{f}\|_{H^1_z(\Omega;{\mathbb{R}^n})}$.
This finishes the proof of the corollary.
\end{proof}

\section{Proof of Theorem \ref{th1.2}}\label{s6}
In this section, we prove Theorem \ref{th1.2} by using Theorem \ref{th1.1} and a perturbation method.
We begin with recalling the unique solvability of the Robin problem \eqref{e1.12} when $p=2$.

\begin{remark}\label{r6.1}
Similarly to Remark \ref{r1.1}, by the Lax--Milgram theorem and the Friedrichs inequality
(see, for instance, \cite[Section 1.1.8, Theorem 1.9]{j13} and \cite[Theorem 6.1]{ls04}),
we conclude that, when $p=2$, the Robin problem \eqref{e1.12} with $\boldsymbol{f}\in L^2(\Omega;{\mathbb{R}^n})$
is uniquely solvable and the weak solution $u$ satisfies
$$
\|u\|_{W^{1,2}(\Omega)}\lesssim\|\boldsymbol{f}\|_{L^2(\Omega;{\mathbb{R}^n})}
$$
with the implicit positive constant independent of $u$ and $f$ (see, for instance, \cite[Remark 1.2]{yyy21}).
\end{remark}

\begin{lemma}\label{l6.1}
Let $n\ge2$, $\Omega\subset{\mathbb{R}^n}$ be a bounded Lipschitz domain, and $\beta$ be the same as in \eqref{e1.11}.
Assume that $A$ and $\Omega$ satisfy Assumption \ref{a1.1}. Let $p\in(1,\infty)$.
Then the Robin problem \eqref{e1.12} with $\boldsymbol{f}\in L^p(\Omega;{\mathbb{R}^n})$ is uniquely solvable
and the weak solution $u$ satisfies
\begin{equation}\label{e6.1}
\|u\|_{W^{1,p}(\Omega)}\le C\|\boldsymbol{f}\|_{L^p(\Omega;{\mathbb{R}^n})},
\end{equation}
where $C$ is a positive constant independent of $u$ and $\boldsymbol{f}$.
\end{lemma}

The proof of Lemma \ref{l6.1} is similar to that of \cite[Theorem 2.6]{dl21}; we omit the details.

Now, we show Theorem \ref{th1.2} by using Lemma \ref{l6.1} and Theorem \ref{th1.1}.

\begin{proof}[Proof of Theorem \ref{th1.2}]
We first prove (i). Let $\boldsymbol{f}\in \mathrm{BMO}_r(\Omega;{\mathbb{R}^n})$ and $u$ be the weak
solution to the Robin problem \eqref{e1.12}. By the fact that $\mathrm{BMO}_r(\Omega)\subset L^p(\Omega)$
with $p\in(n,\infty)$, Lemmas \ref{l2.1}, \ref{l2.2}, and \ref{l6.1}, and the Sobolev inequality,
we have $u\in L^\infty(\Omega)$ and
\begin{align}\label{e6.5}
\|u\|_{L^\infty(\Omega)}&\lesssim\|u\|_{W^{1,p}(\Omega)}\lesssim\|\boldsymbol{f}\|_{L^p(\Omega;{\mathbb{R}^n})}
\lesssim\left\|\boldsymbol{f}-(\boldsymbol{f})_\Omega\right\|_{L^p(\Omega;{\mathbb{R}^n})}+\|\boldsymbol{f}\|_{L^1(\Omega;{\mathbb{R}^n})}\\ \notag
&\lesssim\|\boldsymbol{f}\|_{\mathrm{BMO}(\Omega;{\mathbb{R}^n})}+\|\boldsymbol{f}\|_{L^2(\Omega;{\mathbb{R}^n})}\lesssim
\|\boldsymbol{f}\|_{\mathrm{BMO}_{r,+}(\Omega;{\mathbb{R}^n})}.
\end{align}
Let $v$ be a weak solution of the Neumann problem
\begin{equation*}
\left\{\begin{array}{ll}
\mathrm{div}(A\nabla v)=\mathrm{div}{\boldsymbol{f}}\ \ &\text{in}\ \Omega,\\
\dfrac{\partial v}{\partial\boldsymbol{\nu}}=\boldsymbol{f}\cdot\boldsymbol{\nu}\  &\text{on}\  \partial\Omega
\end{array}\right.
\end{equation*}
and $w:=u-v$. Then $w$ is a weak solution of the Neumann problem
\begin{equation}\label{e6.5a}
\left\{\begin{array}{ll}
\mathrm{div}(A\nabla w)=0\ \ &\text{in}\ \Omega,\\
\dfrac{\partial w}{\partial\boldsymbol{\nu}}=-\beta u\  &\text{on}\  \partial\Omega.
\end{array}\right.
\end{equation}
It is worth pointing out that the condition $\int_{\partial\Omega} \beta u\,d\sigma(x)=0$ is necessary for the solvability
of the Neumann problem \eqref{e6.5a}, and $\int_{\partial\Omega} \beta u\,d\sigma(x)=0$ follows from \eqref{e1.13}.
From Theorem \ref{th1.1}, we deduce that
\begin{align}\label{e6.6}
\|\nabla v\|_{\mathrm{BMO}_{r,+}(\Omega;{\mathbb{R}^n})}\lesssim\|\boldsymbol{f}\|_{\mathrm{BMO}_{r,+}(\Omega;{\mathbb{R}^n})}.
\end{align}
Now let $u_1$ be the weak solution of the Neumann problem \eqref{e1.6} with the coefficient matrix
$A^t$ and the right-hand side $\boldsymbol{f}_1\in H^1_z(\Omega;{\mathbb{R}^n})\cap L^2(\Omega;{\mathbb{R}^n})$, satisfying $\int_\Omega u_1\,dx=0$. By the Sobolev trace theorem (see \cite[Section 2.4.2, Theorem 4.2]{j13})
and Corollary \ref{c1.1}, we get
\begin{align}\label{e6.7}
\|u_1\|_{L^1(\partial\Omega)}\lesssim\|u_1\|_{W^{1,1}(\Omega)}\lesssim\|\nabla u_1\|_{L^1(\Omega;{\mathbb{R}^n})}
\lesssim\|\boldsymbol{f}_1\|_{H^1_z(\Omega;{\mathbb{R}^n})}.
\end{align}
Moreover, we have
\begin{align*}
\int_{\Omega}\boldsymbol{f}_1(x)\cdot\nabla w(x)\,dx=\int_{\Omega}A^t(x)\nabla u_1(x)\cdot \nabla w(x)\,dx
=-\int_{\partial\Omega}\beta(x)u(x)u_1(x)\,d\sigma(x),
\end{align*}
which, together with Lemma \ref{l2.3}, \eqref{e6.5}, and \eqref{e6.7}, further implies that
\begin{align}\label{e6.8}
\|\nabla w\|_{\mathrm{BMO}_r(\Omega;{\mathbb{R}^n})}&\sim\sup_{\|\boldsymbol{f}_1\|_{H^1_z(\Omega;{\mathbb{R}^n})}\le1}
\left|\int_{\Omega}\boldsymbol{f}_1(x)\cdot\nabla w(x)\,dx\right|\\ \nonumber
&\sim\sup_{\|\boldsymbol{f}_1\|_{H^1_z(\Omega;{\mathbb{R}^n})}\le1}\left|\int_{\partial\Omega}\beta(x)u(x)u_1(x)\,d\sigma(x)\right|\\ \nonumber
&\lesssim\sup_{\|\boldsymbol{f}_1\|_{H^1_z(\Omega;{\mathbb{R}^n})}\le1}
\|u\|_{L^\infty(\Omega)}\|u_1\|_{L^1(\partial\Omega)}\\ \nonumber
&\lesssim\sup_{\|\boldsymbol{f}_1\|_{H^1_z(\Omega;{\mathbb{R}^n})}\le1}\|\boldsymbol{f}\|_{\mathrm{BMO}_{r,+}(\Omega;{\mathbb{R}^n})}
\|\boldsymbol{f}_1\|_{H^1_z(\Omega;{\mathbb{R}^n})}\le\|\boldsymbol{f}\|_{\mathrm{BMO}_{r,+}(\Omega;{\mathbb{R}^n})}.
\end{align}
Similarly, we also have
$$
\|\nabla w\|_{L^2(\Omega;{\mathbb{R}^n})}\lesssim\|\boldsymbol{f}\|_{\mathrm{BMO}_{r,+}(\Omega;{\mathbb{R}^n})}.
$$
From this, \eqref{e6.8}, \eqref{e6.6}, and the fact that $\nabla u=\nabla v+\nabla w$, it follows that
$\nabla u\in \mathrm{BMO}_r(\Omega;{\mathbb{R}^n})$ and $\|\nabla u\|_{\mathrm{BMO}_{r,+}(\Omega;{\mathbb{R}^n})}
\lesssim\|\boldsymbol{f}\|_{\mathrm{BMO}_{r,+}(\Omega;{\mathbb{R}^n})}$. This finishes the proof of (i).
The proof of (ii) is similar to that of Corollary \ref{c1.1} and we omit the details here.
This finishes the proof of Theorem \ref{th1.2}.
\end{proof}

\medskip

\noindent\textbf{Data Availability Statement}\quad Data sharing is not applicable to this article as no
data sets were generated or analysed.

\medskip

\noindent\textbf{Conflict of interest}\quad All authors state no conflict of interest.

\bigskip

\noindent Hongjie Dong (Corresponding author)

\medskip

\noindent Division of Applied Mathematics, Brown University, Providence RI 02912, USA

\smallskip

\noindent {\it E-mail}: \texttt{Hongjie\_Dong@Brown.edu}

\bigskip

\noindent Dachun Yang

\medskip

\noindent Laboratory of Mathematics and Complex Systems (Ministry of Education of China),
School of Mathematical Sciences, Beijing Normal University, Beijing 100875,
People's Republic of China

\smallskip

\noindent {\it E-mail}: \texttt{dcyang@bnu.edu.cn}

\bigskip

\noindent Sibei Yang
\medskip

\noindent School of Mathematics and Statistics, Gansu Key Laboratory of Applied Mathematics
and Complex Systems, Lanzhou University, Lanzhou 730000, People's Republic of China

\smallskip

\smallskip

\noindent {\it E-mail}: \texttt{yangsb@lzu.edu.cn}

\hspace{0.888cm}


\begin{thebibliography}{99}

\bibitem{a92}P. Acquistapace, On BMO regularity for linear elliptic systems, Ann. Mat. Pura Appl. (4) 161 (1992),
231--269.
\vspace{-0.3cm}

\bibitem{acmm10} A. Alvino, A. Cianchi, V. G. Maz'ya and A. Mercaldo, Well-posed elliptic Neumann
problems involving irregular data and domains, Ann. Inst. H. Poincar\'e Anal. Non
 Lin\'eaire 27 (2010), 1017--1054.

\vspace{-0.3cm}

\bibitem{aq02}P. Auscher and M. Qafsaoui, Observations on $W^{1,p}$ estimates for divergence elliptic equations with
$\mathrm{VMO}$ coefficients, Boll. Unione Mat. Ital. Sez. B Artic. Ric. Mat. (8) 5 (2002), 487--509.

\vspace{-0.3cm}
\bibitem{bg18}M. Bolkart, Y. Giga, T. Suzuki and Y. Tsutsui, Equivalence of $\mathrm{BMO}$-type norms with
applications to the heat and Stokes semigroups, Potential. Anal. 49 (2018), 105--130.

\vspace{-0.3cm}
\bibitem{bc22} D. Breit, A. Cianchi,  L. Diening and S. Schwarzacher, Global Schauder estimates for the
$\mathrm{p}$-Laplace system, Arch. Rational Mech. Anal. 243 (2022), 201--255.

\vspace{-0.3cm}
\bibitem{bk95} S. Buckley and P. Koskela, Sobolev--Poincar\'e implies John,
Math. Res. Lett. 2 (1995), 577--593.

\vspace{-0.3cm}
\bibitem{b05} S.-S. Byun, Elliptic equations with BMO coefficients in Lipschitz domains,
Trans. Amer. Math. Soc. 357 (2005), 1025--1046.

\vspace{-0.3cm}
\bibitem{bw05} S.-S. Byun and L. Wang, The conormal derivative problem for elliptic equations with BMO coefficients
on Reifenberg flat domains, Proc. Lond. Math. Soc. (3) 90 (2005), 245--272.

\vspace{-0.3cm}
\bibitem{bw04} S.-S. Byun and L. Wang, Elliptic equations with BMO coefficients
in Reifenberg domains, Comm. Pure Appl. Math. 57 (2004), 1283-1310.

\vspace{-0.3cm}
\bibitem{C94} D.-C. Chang, The dual of Hardy spaces on a bounded domain in ${\mathbb{R}^n}$, Forum Math. 6 (1994), 65--81.

\vspace{-0.3cm}

\bibitem{cds99} D.-C. Chang, G. Dafni and E. M. Stein, Hardy spaces, BMO, and boundary
value problems for the Laplacian on a smooth domain in ${\mathbb{R}^n}$,
Trans. Amer. Math. Soc. 351 (1999), 1605--1661.

\vspace{-0.3cm}

\bibitem{cks93} D.-C. Chang, S. G. Krantz and E. M. Stein, $H^p$ theory on a smooth
domain in $R^N$ and elliptic boundary value problems, J. Funct. Anal. 114 (1993), 286--347.

\vspace{-0.3cm}
\bibitem{cjy16} X. Chen, R. Jiang and D. Yang, Hardy and Hardy--Sobolev spaces on strongly Lipschitz domains
and some applications, Anal. Geom. Metr. Spaces 4 (2016), 336--362.

\vspace{-0.3cm}
\bibitem{cm14} A. Cianchi and V. G. Maz'ya, Global boundedness of the gradient for a class of nonlinear
elliptic systems, Arch. Ration. Mech. Anal. 212 (2014), 129--177.

\vspace{-0.3cm}

\bibitem{d96} G. Di Fazio, $L^p$ estimates for divergence form elliptic equations with
discontinuous coefficients, Boll. Un. Mat. Ital. A (7) 10 (1996), 409--420.

\vspace{-0.3cm}
\bibitem{dks12}L. Diening, P. Kaplick\'y and P. Schwarzacher, BMO estimates for the $\mathrm{p}$-Laplacian,
Nonlinear Anal. 75 (2012), 637--650.

\vspace{-0.3cm}
\bibitem{d20} H. Dong, Recent progress in the $L^p$ theory for elliptic and parabolic equations with
discontinuous coefficients, Anal. Theory Appl. 36 (2020), 161--199.

\vspace{-0.3cm}
\bibitem{dek18}H. Dong, L. Escauriaza and S. Kim, On $C^1,C^2$, and weak type-$(1,1)$ estimates for linear elliptic
operators: part II, Math. Ann. 370 (2018), 447--489.

\vspace{-0.3cm}

\bibitem{dk18} H. Dong and D. Kim, On $L^p$-estimates for elliptic and parabolic equations
with $A_p$ weights, Trans. Amer. Math. Soc. 370 (2018), 5081--5130.

\vspace{-0.3cm}

\bibitem{dk12} H. Dong and D. Kim, The conormal derivative problem for higher order
elliptic systems with irregular coefficients, In: Recent Advances in Harmonic Analysis
and Partial Differential Equations, pp.\,69--97, Contemp. Math. 581, Amer. Math. Soc., Providence, RI, 2012.

\vspace{-0.3cm}

\bibitem{dk11} H. Dong and D. Kim, On the $L^p$-solvability of higher order parabolic and elliptic systems
with BMO coefficients, Arch. Ration. Mech. Anal. 199 (2011), 889--941.

\vspace{-0.3cm}
\bibitem{dk10}H. Dong and D. Kim, Elliptic equations in divergence form with partially BMO coefficients,
Arch. Ration. Mech. Anal. 196 (2010), 25--70.

\vspace{-0.3cm}
\bibitem{dk17}H. Dong and S. Kim,  On $C^1,C^2$, and weak type-$(1,1)$ estimates for linear elliptic
operators, Comm. Partial Differential Equations 42 (2017), 417--435.

\vspace{-0.3cm}
\bibitem{dlk20}H. Dong, J. Lee and S. Kim, On conormal and oblique derivative problem for elliptic equations
with Dini mean oscillation coefficients, Indiana Univ. Math. J. 69 (2020), 1815--1853.

\vspace{-0.3cm}
\bibitem{dl21} H. Dong and Z. Li, The conormal and Robin boundary value problems in nonsmooth domains
satisfying a measure condition, J. Funct. Anal. 281 (2021), Paper No. 109167, 32 pp.

\vspace{-0.3cm}
\bibitem{g12}J. Geng, $W^{1,p}$ estimates for elliptic problems with Neumann boundary conditions in Lipschitz
domains, Adv. Math. 229 (2012), 2427--2448.

\vspace{-0.3cm}
\bibitem{gg22}Y. Giga and Z. Gu, The Helmholtz decomposition of a space of vector fields with bounded mean
oscillation in a bounded domain, Math. Ann. 386 (2023), 673--712.

\vspace{-0.3cm}
\bibitem{g14a} L. Grafakos, Modern Fourier Analysis,
third edition, Graduate Texts in Mathematics 250, Springer, New York, 2014.

\vspace{-0.3cm}
\bibitem{hl}Q. Han and F. Lin, Elliptic Partial Differential Equations, Second edition,
Courant Lecture Notes in Mathematics 1, Courant Institute of Mathematical Sciences,
New York, American Mathematical Society, Providence, RI, 2011.

\vspace{-0.3cm}
\bibitem{i98}T. Iwaniec, The Gehring lemma, In: Quasiconformal Mappings and Analysis (Ann Arbor, MI, 1995),
pp.\,181--204, Springer, New York, 1998.

\vspace{-0.3cm}
\bibitem{j76}S. Janson, On functions with conditions on the mean oscillation, Ark. Mat. 14 (1976), 189--196.

\vspace{-0.3cm}
\bibitem{j80} P. W. Jones, Extension theorems for BMO, Indiana Univ. Math. J. 29 (1980), 41--66.

\vspace{-0.3cm}

\bibitem{k94} C. E. Kenig, Harmonic Analysis Techniques for Second Order Elliptic
Boundary Value Problems, CBMS Regional Conference Series in Mathematics 83,
American Mathematical Society, Providence, RI, 1994.

\vspace{-0.3cm}

\bibitem{k07} N. V. Krylov, Parabolic and elliptic equations with VMO coefficients,
Comm. Partial Differential Equations 32 (2007), 453--475.

\vspace{-0.3cm}
\bibitem{ls04} L. Lanzani and Z. Shen, On the Robin boundary condition for Laplace's
equation in Lipschitz domains, Comm. Partial Differential Equations 29 (2004), 91--109.

\vspace{-0.3cm}

\bibitem{l17} Y. Li, On the $C^1$ regularity of solutions to divergence form elliptic
systems with Dini-continuous coefficients, Chinese Ann. Math. Ser. B 38 (2017), 489--496.

\vspace{-0.3cm}
\bibitem{n08}E. Nakai, A generalization of Hardy spaces $H^p$ by using atoms, Acta Math. Sin. (Engl. Ser.)
24 (2008), 1243--1268.

\vspace{-0.3cm}
\bibitem{n97}E. Nakai, Pointwise multipliers on weighted BMO spaces, Studia Math. 125 (1997), 35--56.

\vspace{-0.3cm}
\bibitem{ny97}E. Nakai and K. Yabuta, Pointwise multipliers for functions of weighted bounded mean oscillation on
spaces of homogeneous type, Math. Japon. 46 (1997), 15--28.

\vspace{-0.3cm}
\bibitem{ny85}E. Nakai and K. Yabuta, Pointwise multipliers for functions of bounded mean oscillation,
J. Math. Soc. Japan 37 (1985), 207--218.

\vspace{-0.3cm}

\bibitem{j13} J. Ne$\mathrm{\check{c}}$as, Direct Methods in the Theory of Elliptic Equations,
Translated from the 1967 French original by Gerard Tronel and Alois Kufner,
Editorial coordination and preface by $\mathrm{\check{S}}$\'arka Ne$\mathrm{\check{c}}$asov\'a
and a contribution by Christian G. Simader, Springer Monographs in Mathematics, Springer, Heidelberg, 2012.

\vspace{-0.3cm}

\bibitem{s75} D. Sarason, Functions of vanishing mean oscillation, Trans. Amer. Math. Soc.
207 (1975), 391--405.

\vspace{-0.3cm}
\bibitem{sh18} Z. Shen, Periodic Homogenization of Elliptic Systems,
Operator Theory: Advances and Applications 269, Advances in Partial
Differential Equations (Basel), Birkh\"auser/Springer, Cham, 2018.

\vspace{-0.3cm}
\bibitem{sh05a} Z. Shen, Bounds of Riesz transforms on $L^p$ spaces for second order
elliptic operators, Ann. Inst. Fourier (Grenoble) 55 (2005), 173--197.

\vspace{-0.3cm}

\bibitem{St93} E. M. Stein, Harmonic Analysis: Real-variable
Methods, Orthogonality, and Oscillatory Integrals, Princeton
University Press, Princeton, NJ, 1993.

\vspace{-0.3cm}
\bibitem{ycyy20} S. Yang, D.-C. Chang, D. Yang and W. Yuan, Weighted gradient estimates for elliptic problems
with Neumann boundary conditions in Lipschitz and (semi-)convex domains, J. Differential Equations
268 (2020), 2510--2550.

\vspace{-0.3cm}
\bibitem{yyy21} S. Yang, D. Yang and W. Yuan, Weighted global regularity estimates for elliptic problems with
Robin boundary conditions in Lipschitz domains, J. Differential Equations 296 (2021), 512--572.

\end{thebibliography}
\end{document}